%% file: ms.tex
\documentclass[final,onefignum,onetabnum]{siamart171218}

\input{ex_shared}

\ifpdf
\hypersetup{
  pdftitle={Load-dependent machine failures in production network models},
  pdfauthor={S.\ G\"ottlich, S.\ Knapp}
}
\fi

\begin{document}

\maketitle

% REQUIRED
\begin{abstract}
In this paper, a production model based on (hyperbolic) differential equations with stochastic and load-dependent machine failures is introduced. We derive the model on the base of a well-established deterministic model and show its well-posedness. 
To do so, we make use of the theory of piecewise deterministic Markov processes and fuse it with the theory of the underlying deterministic production model.
Finally, we compare the load-dependent model to the already established load-independent model and highlight the new properties in numerical examples.
\end{abstract}

% REQUIRED

% REQUIRED
%60J20   	Applications of Markov chains and discrete-time Markov processes on general state spaces (social mobility, learning theory, industrial processes, etc.) 
 	%60J25   	Continuous-time Markov processes on general state spaces
 	%60J22   	Computational methods in Markov chains 
 	%90B30   	Production models
 	%90B25   	Reliability, availability, maintenance, inspection
 	%35L65   	Conservation laws
 	%65C05   	Monte Carlo methods
 	%65C20   	Models, numerical methods
%%%%%%%%%%%%%%%%%%%%%%%%%%%%%%%%%%%%%%%%%%%%%%%%%%%%%%%%%%%%%%%%%%%%%%%%%%%%%%%%%%%%%%%%%%%%%%%%%%%%%%%%%%%%%%%%%%%%%%%%%%%%%%%%%%%%%%%%%%%%%%%%%%%%%%%%%%%%%%%%%%%
%%%%%%%%%%%%%%%%%%%%%%%%%%%%%%%%%%%%%%%%%%%%%%%%%%%%%%%%%%%%%%%%%%%%%%%%%%%%%%%%%%%%%%%%%%%%%%%%%%%%%%%%%%%%%%%%%%%%%%%%%%%%%%%%%%%%%%%%%%%%%%%%%%%%%%%%%%%%%%%%%%%
\section{Introduction}
Mathematical models for production systems as well as supply chains are a recent research topic and a variety of modeling approaches are taken into account. 
Most of these models are either based on discrete event simulations \cite{ArmbrusterDegondRinghofer2006,BanksCarson} or Newton-type dynamics~\cite{Gottlich2015} resulting in a microscopic production model.
Macroscopic production models naturally arise from the microscopic production models and are based on ordinary differential equations (ODEs) \cite{GoettlichHertyRinghofer}, hyperbolic partial differential equations (PDEs) \cite{ArmbrusterDegondRinghofer2006,DApiceManzoPiccoli2009,DApiceManzoPiccoli2010,DApiceManzoPiccoli2013,Forestier2015} or a mixture of both \cite{DApiceKogutManzo2014,DApiceManzoPiccoli2012,GoettlichHertyKlar2005}. 
In \cite{ApiceGoettlichHertyBenedetto} a comprehensive overview of macroscopic production models is given. We focus on the macroscopic production network model from \cite{ApiceGoettlichHertyBenedetto,GoettlichMartinSickenberger}, where the deterministic dynamics is given by a coupled system of hyperbolic PDEs and ODEs.

In \cite{DegondRinghofer2007,GoettlichKnapp2017,GoettlichMartinSickenberger}, stochastic effects are introduced into macroscopic production models, where externally given stochastic capacity functions to model machine failures (or capacity drops) are used.
The randomness in capacities strongly influences the dynamics of the production and leads to interesting system behavior.
Up to now, the capacity of a production step is determined by, e.g., machine restrictions or the number of workers and the assumption that machine failures are independent of the production process. 
The latter assumption is quite restrictive since a high workload implies a high abrasion of machines, or stressful working conditions for individuals lead to more sickness, whereas an empty production is not affected by machine failures. 
This motivates to introduce an influence of the production on the machine failure probabilities as well, and we obtain a bidirectional relation between the deterministic production and the random machine failures.

Different to existing approaches \cite{DegondRinghofer2007,GoettlichKnapp2017,GoettlichMartinSickenberger}, where the capacities are stochastic processes inserted into the deterministic production model as capacity functions, we have to consider the deterministic production dynamics and the random capacity functions, simultaneously. 

From the mathematical point of view, we consider the production together with the capacity process as a whole stochastic process. The theory of piecewise deterministic Markov processes, which is  well developed in e.g.\ \cite{Davis1984,Jacobsen2006,Saporta2015} is the key to show well-posedness. 
In most applications of piecewise deterministic Markov processes, the deterministic dynamics is given by a system of ODEs \cite{Alkurdi2013,Davis1984,Saporta2015} or by parabolic PDEs \cite{Alkurdi2013}. In this context, the state space is ``nice'', i.e.\ a Borel space but in the context of hyperbolic conservation laws the solution space is not a Borel space in general.
Since we consider a system of coupled PDEs and ODEs, where the PDEs are of hyperbolic type, we face the difficulty that the natural space for scalar hyperbolic PDEs, i.e., the space of functions with bounded total variation forms no Borel space, and the standard theory of piecewise deterministic Markov processes with general state space fails.
Luckily, the semigroup of the deterministic production network can be extended to a semigroup on a Borel space such that we workaround this issue by using the extended semigroup to construct the stochastic process.
%There exist recent results to simulate this type of stochastic processes in e.g.\ \cite{Lemaire2017,Saporta2015}  

This manuscript is organized as follows: in section \ref{sec:ModEq}, we introduce the base deterministic production network model and extend it to a Markovian load-dependent production network model in a second step. This section is followed by section \ref{sec:NumTreat}, where the numerical treatment of the load-dependent model is introduced and applied to an example, where we highlight the differences and similarities of the load-independent and the load-dependent model numerically.

%%%%%%%%%%%%%%%%%%%%%%%%%%%%%%%%%%%%%%%%%%%%%%%%%%%%%%%%%%%%%%%%%%%%%%%%%%%%%%%%%%%%%%%%%%%%%%%%%%%%%%%%%%%%%%%%%%%%%%%%%%%%%%%%%%%%%%%%%%%%%%%%%%%%%%%%%%%%%%%%%%%
%%%%%%%%%%%%%%%%%%%%%%%%%%%%%%%%%%%%%%%%%%%%%%%%%%%%%%%%%%%%%%%%%%%%%%%%%%%%%%%%%%%%%%%%%%%%%%%%%%%%%%%%%%%%%%%%%%%%%%%%%%%%%%%%%%%%%%%%%%%%%%%%%%%%%%%%%%%%%%%%%%%
\section{Modeling Equations}\label{sec:ModEq}
This section briefly recalls the base deterministic model from \cite{ApiceGoettlichHertyBenedetto,GoettlichHertyKlar2005} in the first subsection \ref{subsec:IntroDet}, which is extended to a load-dependent stochastic production network model in the subsequent subsection \ref{subsec:LoadDepModel}.
\subsection{Brief Introduction to the Deterministic Model}
Let  $G = (\cV,\cC)$ be a directed graph consisting of a set $\cC = \{1,\dots,N\}$ of arcs, where $N \in \N$ and a non empty set of vertices $\cV$. We interpret every arc as a processor equipped with a storage or queue, respectively, in front of it. According to \cite{ApiceGoettlichHertyBenedetto}, the processor $e \in \cC$ is characterized by the length $L^e$, which is mapped on the interval $[a^e,b^e] \subset \R$. Therefore, the queue is located at $x = a^e$, directly at the corresponding vertex $s(e) \in \cV$. Additionally, a processing velocity $v^e > 0$ and a time-dependent capacity $\mu^e(t) \geq 0$ are given user-defined parameters.
As usual in graph theory, we denote by $\delta_v^-$ and $\delta_v^+$ the set of all ingoing and outgoing arcs for every vertex $v \in \cV$.
At all inflow vertices $v \in V_{\text{in}} = \{v \in \cV \colon \delta_v^-  = \emptyset\}$ a time-dependent inflow function $G^v_{\text{in}}(t)$ is prescribed and for every $v \in \cV$ with $|\delta_v^+|>0$ we have time-dependent distribution rates $A^{v,e}(t) \in [0,1], e \in \delta_v^+$ satisfying $\sum_{e \in \delta_v^+} A^{v,e} (t) = 1$.
We call the directed graph $G$ together with the described properties and parameters a \emph{production network} in the following. 
We now further specify the notation of a deterministic production network model:
\begin{definition}[Deterministic production network model]\label{def:det_model}
Let $G = (\cV,\cC)$ be a production network. The \emph{deterministic production network model} is defined by the following equations:
 \begin{align}
\partial_t\rho^e(x,t)+\partial_xf^e(t,\rho^e(x,t))&= 0 ,\label{eq:det_net_1}\\
f^e(t,\rho^e(x,t))&= \min\{v^e\rho^e(x,t),\mu^e(t)\},\label{eq:det_net_2}\\
\rho^e(x,0) &= \rho^e_0(x)\text{,}\label{eq:det_net_3}\\
v^e\rho^e(a^e,t) &= g_{\text{out}}^e(t), \label{eq:det_net_4}\\
\partial_t q^e(t) &= g^e_{\text{in}}(t)-g_{\text{out}}^e(t)\text{,}\notag\\
q^e(0) &= q_0^e \notag
\end{align}
for $x\in (a^e,b^e)$ with 
\begin{align*}
g_{\text{in}}^e(t) &= 
\begin{cases}
A^{s(e),e}(t)\sum_{\tilde e \in \delta_{s(e)}^-} f^{\tilde e}(t,\rho^{\tilde e}(b^{\tilde e},t)) &\text{ if } s(e) \notin V_{\text{in}},\\
G_{\text{in}}^{s(e)}(t) &\text{ if } s(e) \in V_{\text{in}}\text{,}
\end{cases}
\end{align*}
and 
\begin{align*}
g_{\text{out}}^e(t) &= 
\begin{cases}
\min\{g_{\text{in}}^e(t),\mu^e(t)\} &\text{ if } q^e(t) = 0,\\
\mu^e(t) &\text{ if } q^e(t)>0
\end{cases}
\end{align*}
for all $e\in \cC$, $t\in [0,T]$, $T>0$ and given initial conditions $\rho^e_0$ and $q^e_0$. 
\end{definition}

From the mathematical point of view, the deterministic production network model is a coupled system of PDEs and ODEs, where the PDEs are scalar hyperbolic conservation laws. Total variation plays a key role in the theory of hyperbolic conservation laws, and we define, as in \cite{Walter1987}, the total variation of a function $f \colon I \to \R$ on an interval $I \subset \R$ as 
\begin{align*}
\TV_{I}(f)  = \sup\left\{\sum_{i = 1}^N |f(x_i)-f(x_{i-1})| \colon  x_0 < x_1 <\cdots <x_N \in I, N \in \N\right\}.
\end{align*}

If $f \colon (a,b)\times (0,T) \to \R$ is a function of two variables, we use the Tonelli-Cesari variation (cf.~\cite{Cesari1936}), which is given by
\begin{align*}
V_{\text{TC}}(f) &= \inf\{V_{\text{T}}(g) \colon g = f \text{ almost everywhere}\} \text{ with}\\
V_{\text{T}}(g) &= \int_a^b \TV_{[0,T]}(g(x,\cdot))dx+\int_0^T \TV_{[a,b]}(g(\cdot,t))dt.
\end{align*}
We denote by $\BV((a,b))$ and $\BV((a,b)\times (0,T))$ the sets of functions $f$ with bounded total variation $\TV_{(a,b)}(f)<\infty$ and bounded variation $V_{\text{TC}}(f)<\infty$, respectively.

Since the partial differential equations of the deterministic production network model are of hyperbolic type, we need the concept of weak entropy solutions to define a network solution later on.
The following definition of a weak entropy solution is taken from \cite{Bustos1999} and adapted to the deterministic production network model.
\begin{definition}[Weak entropy solution]
A function $\rho \in \BV((a^e,b^e)\times (0,T))$ is a \emph{weak entropy solution} to \eqref{eq:det_net_1}-\eqref{eq:det_net_4} if for almost all $x \in (a^e,b^e)$ we have $\rho(x,0) = \rho_0(x)$ and for all $k \in \R$ and all $\phi \in C_0^2([a^e,b^e) \times (0,T))$ with $\phi \geq 0$, it holds
\begin{align*}
&\; \int_0^T \int_{a^e}^{b^e} |\rho(x,t)-k| \phi_t(x,t)+\sgn(\rho(x,t)-k)|f^e(t,\rho(x,t))-f^e(t,k)| \phi_x(x,t) dx dt \\
&\;\geq -\int_0^T \sgn\left(\frac{g_{\text{out}}^e(t)}{v^e}-k\right)|f^e(t,\lim_{\substack{x \searrow a^e \\ x \notin B}} \rho(x,t))-f^e(t,k)| \phi(a^e,t)dt,
\end{align*}
where $B$ is a set of measure zero, $\sgn(x)$ is the sign function, and $C_0^2([a^e,b^e) \times (0,T))$ is the set of all twice continuously differentiable functions with compact support in $[a^e,b^e) \times (0,T)$.
\end{definition}

Thus, we are able to define a solution of the production network model as follows: 
\begin{definition}[Network solution]\label{def:ProdNetSol}
$\cX(t) = (\rho^1(t),q^1(t),\dots,\rho^N(t),q^N(t))$ is a \emph{network solution} of the deterministic production network model from definition \ref{def:det_model} if for every $e = 1, \dots,N$ the functions $\rho^e$ are weak entropy solutions to \eqref{eq:det_net_1}-\eqref{eq:det_net_4} and 
\begin{align*}
q^e(t) = q^e_0 + \int_0^t (g_{\text{in}}^e(s)-g_{\text{out}}^e(s)) ds
\end{align*}
holds for every $t \in [0,T]$.
\end{definition}

Since a weak entropy solution is of bounded variation, we expect conditions on the initial and boundary values to guarantee the well-posedness of a network solution, as the following theorem \ref{thm:ExistenceUniquenessDetProd} collected from \cite{ApiceGoettlichHertyBenedetto} shows.
\begin{theorem}\label{thm:ExistenceUniquenessDetProd}
Let $G = (\cV,\cC)$ be a production network with time-independent capacities and $\TV_{(0,T)}(G_{\text{\emph{in}}}^v)<\infty$ for every $v \in V_{\text{\emph{in}}}$. Then, there exists a semigroup $(S_t, t\in [0,T])$ on 
\[D = \left\{(\rho^1,q^1,\dots,\rho^N,q^N) \colon \rho^e \in L^1((a^e,b^e)), q^e \in \R_{\geq 0}, e = 1,\dots,N \right\}\] equipped with the norm
\[\|(\rho^1,q^1,\dots,\rho^N,q^N)\| = \sum_{e = 1}^N \|\rho^e\|_{L^1((a^e,b^e))}+|q^e|,\]
which satisfies
\begin{enumerate}
\item $S_{s+t}u = S_s(S_t u)$ and $S_0u = u$ for every $u \in D$,
\item the mapping $t \mapsto S_t u$ is continuous,
\item for every $t \in [0,T]$, the operator $S_t$ satisfies $\|S_t u-S_t \tilde{u}\| \leq  \|u-\tilde{u}\|$
for all $u, \tilde{u} \in D$,
\item for every \[u \in \tilde{D} = \{(\rho^1,q^1,\dots,\rho^N,q^N) \in D \colon \TV_{(a^e,b^e)}(\rho^e) <\infty, e = 1, \dots, N\},\] 
the function $t \mapsto S_t u$ is a unique network solution,
\item for every $u \in \tilde{D}$ there exists $\tilde{L}>0$ such that $\|S_t u- S_s u\| \leq \tilde{L} |t-s|$
for every $s,t \in [0,T]$, 
\item $S_t u \in \tilde{D}$ for every $u \in \tilde{D}$.
\end{enumerate}
\end{theorem}
We comment theorem \ref{thm:ExistenceUniquenessDetProd} briefly below. The proof can be found in \cite{ApiceGoettlichHertyBenedetto} and is based on a wave front tracking method. Furthermore, the Lipschitz continuity of the third statement allows the extension of the solution operator $S_t$ on $\tilde{D}$ to the domain $D$ by considering the closure of all functions in $\tilde{D}$ with respect to the norm given above, see also \cite{Bressan2000}.
Following the proof in \cite{ApiceGoettlichHertyBenedetto}, we deduce 
\[\|S_t u- S_s u\| \leq \left(\hat{L}+\sum_{e=1}^N \operatorname{TV}_{(a^e,b^e)}(\rho^e(\cdot,s))\right) |t-s|\]
for $u \in \tilde{D}$, $s<t$ and $\hat{L}>0$ a constant depending on the capacities, the velocities and the external inflows only. Since the Lipschitz constant in the time variable depends on the total variation of the solution, we cannot expect Lipschitz continuity for general data $u \in D$. But one can show the continuity in time, i.e.~statement two of the theorem, directly.
%Let $\xi,t \geq 0$, then we obtain
%\begin{align*}
%\|S_{t+\xi}u-S_tu\|\leq \|S_\xi u-u\|
%\end{align*}
%for all $u \in D$ by statement five of theorem \ref{thm:ExistenceUniquenessDetProd}. Thus, it remains to show the continuity at $t = 0$. Let $\epsilon>0$ and choose a sequence $(u^n, n\in \N)\subset \tilde{D}$, such that $\lim_{n \to \infty} \|u^n-u\| = 0$. Let $n \in \N$ be, such that \[\|u^n-u\|<\frac{\epsilon}{4}. \] Since $u^n \in \tilde{D}$, we have \[\|S_\xi u^n-u^n\| \leq \tilde{L}|\xi|\] by statement five of the theorem. Choose $\delta< \frac{\epsilon}{2 \tilde{L}}$, then
%\begin{align*}
%\|S_\xi u-u\| &= \|S_\xi u-S_\xi u^n+S_\xi u^n-u^n+u^n-u\|\\
%&\leq \|S_\xi u-S_\xi u^n\|+\|S_\xi u^n-u^n\|+\|u^n-u\|\\
%&\leq 2\|u^n-u\| + \tilde{L}|\xi|\\
%&< \frac{\epsilon}{2} + \frac{\epsilon}{2}\\
%&= \epsilon
%\end{align*}  
%for all $\xi<\delta$.
The following remark is essential for subsection \ref{subsec:LoadDepModel} and imposes why we patiently introduced the semigroup on $D$.
\begin{remark}\label{rem:PolishSpace}
The space $D$ with the norm given in theorem \ref{thm:ExistenceUniquenessDetProd} is a Polish space, i.e.~it is a Borel space with the $\sigma$-algebra generated by the topology induced by the norm. In detail, $L^1((a^e,b^e))$ and $\R$ are complete and separable normed vector spaces with respect to the $L^1$ and Euclidean Norm, respectively. The Cartesian product of countable many Polish spaces with the product topology is again a Polish space, see \cite{Dudley2002}. Since we have only finitely many processors, the space $D$ is a Polish space. However, the space $\tilde{D}$ is no Polish space since it is not complete with respect to the presented norm. If we apply the total variation norm, the space is even not separable. Thus, we are not able to guarantee the existence of regular conditional probabilities on $\tilde{D}$.
\end{remark}

\label{subsec:IntroDet}
\subsection{Load-dependent Model}
\label{subsec:LoadDepModel}
We introduce machine failures in the deterministic production network model by using random and piecewise-constant capacity processes to incorporate machine failures, see e.g. \cite{DegondRinghofer2007,GoettlichKnapp2017, GoettlichMartinSickenberger}.
The major benefit compared to the previous works is that we allow machine failure probabilities, which might depend on the densities and queue-lengths and hence induce a bidirectional relation between the production process and the machine failure probabilities. We consider the production dynamics coupled to the machine failures as a whole stochastic process, where the theory of piecewise deterministic Markov processes taken from \cite{Jacobsen2006} provides the essential tools.

In the following, we assume a production network with $N$ queue-processor units. To ease the notation, we reorder the set $D$ and consider
\begin{align*}
D = \R_{\geq 0}^N \times \bigtimes_{e= 1}^N L^1((a^e,b^e)).
\end{align*} 
Let $S_{uv}^\mu \colon D \to D$ be the semigroup with capacities $\mu = (\mu^1,\dots,\mu^N) \in \R_{\geq 0}^N$ from theorem \ref{thm:ExistenceUniquenessDetProd} starting from $u \in [0,T]$ and with $v \in [u,T]$. 
For every processor, we introduce the state values of the capacities as $\vec{r} = (r_1,\dots,r_N)$, and to capture the complete dynamics of the process, we set the state space
\[
E = \bigtimes_{e = 1}^N \{1,\dots,C^e\}\times D,
\]
where $C^1,\dots,C^N \in \N$ denote the possible number of capacities of processors $1,\dots,N$ here. The state space $E$ is equipped with the $\sigma$-algebra $\cE$ and to map from $\vec{r}$ to the realized capacity, we introduce the function
\begin{align*}
\mu \colon \bigtimes_{e = 1}^N \{1,\dots,C^e\} &\to \R_{\geq 0}^N\\
\vec{r} & \mapsto (\mu^1(r_1),\dots,\mu^N(r_N)).
\end{align*}

We give the following definition for the load-dependent production network model.
\begin{definition}[MLDPNM]\label{def:MLDPNM}
%Let $E = \bigtimes_{e = 1}^N \{1,\dots,C^e\}\times D$ for some $C^1,\dots,C^N \in \N$. 
A \emph{Markovian load-dependent production network model} (MLDPNM) is defined via a stochastic process $X = ((\vec{r}(t),\vec{q}(t),\vec{\rho}(t)), t \in [0,T])$ on some probability space \OAP{} with values in $E$, which satisfies the following conditions:
\begin{enumerate}
\item $X(0) = x_0$ $P$-a.s. for some initial data $x_0 \in E$.
\item $X$ is a Markov process with respect to the natural filtration $\cF^X = (\cF^X_t,t \in [0,T])$.
\item There exist transition rate functions \[\gl^e_{ij} \colon [0,T] \times \R_{\geq 0}\times L^1((a^e,b^e)) \to \R_{\geq 0},\] $i,j = 1,\dots,C^e$ satisfying $\gl^e_{ii} = \sum_{j = 1, j\neq i}^{C^e} \gl^e_{ij}$ such that for every $t \in (0,T)$, $(\vec{r},\vec{q},\vec{\rho}) \in E$ holds 
\begin{align}
P(r^e(t+\Delta t) = j | X(t) = (\vec{r},\vec{q},\vec{\rho})) =&\; \big(1-\Delta t \gl^e_{r^e r^e}(t,q^e,\rho^e)\big)\Ind_{r^e}(j) \notag \\ &\;+\Delta t \gl^e_{r^ej}(t,q^e,\rho^e)(1-\Ind_{r^e}(j)) + \lano(\Delta t)
\end{align}
for every $e = 1,\dots,N$ and for every $j = 1,\dots,C^e$ as $\Delta t \to 0$.
\item There exists a capacity function $\mu \colon \bigtimes_{e = 1}^N \{1,\dots,C^e\} \to \R_{\geq 0}^N$ and a set $\cN \in \cA$ with $P(\cN) = 0$ such that for every $\go \in \cN^c$, there exist times $T_0 = 0\leq T_1 \leq \cdots \leq  T_{M}=T$ such that, for every $k = 0,\dots,M-1$, the capacity $t \mapsto \mu(\vec{r}(t,\go))$ is constant on $[T_k,T_{k+1})$ and 
\[(\vec{q}(t,\omega),\vec{\rho}(t,\omega)) = S^{\mu(\vec{r}(T_k,\go))}_{T_k t}(\vec{q}(T_k,\go),\vec{\rho}(T_k,\go))\] on $[T_k,T_{k+1})$.
\end{enumerate}
\end{definition}

We define the complete deterministic dynamics as 
\begin{align}
\phi_{st} \colon E &\to E,\notag\\
(\vec{r},\vec{q},\vec{\rho}) & \mapsto \phi_{st}(\vec{r},\vec{q},\vec{\rho}) = 
\begin{pmatrix}
\vec{r}\\
S_{st}^{\mu(\vec{r})}(\vec{q},\vec{\rho}) \label{eq:LoadDepPhiNet}
\end{pmatrix}
\end{align}
for every $0\leq s \leq t \leq T$, and we can conclude the following properties of the mapping $\phi_{st}$, which will be a crucial point for showing the existence of an MLDPNM.
\begin{lemma}\hspace{0mm}\label{lem:PropPhi}
\begin{enumerate}
\item We have for every $0 \leq s < t <u\leq T$ and for every $y \in E$ the semigroup property
\[\phi_{su}(y) = \phi_{tu}(\phi_{st}(y)).\]
\item For every $t \in [0,T]$ and $y \in E$, we have $\phi_{tt}(y) = y$.
\item For every $s \in [0,T]$ and $y \in E$, the mapping $t \mapsto \phi_{st}(y)$ is continuous.
\item The mapping $\phi \colon \{(s,t,y) \in [0,T]^2 \times E \colon s \leq t\} \to E$ is continuous and consequently measurable.
\end{enumerate}
\end{lemma}  
\begin{proof}Without loss of generality, we assume $\mu(r) = r$.
The first two statements follow directly from the semigroup property of $S_{st}^r$, i.e., $S^r_{su} = S^r_{tu}S^r_{st}$ and $S^r_{tt} = \operatorname{Id}$, the identity; see theorem \ref{thm:ExistenceUniquenessDetProd}. Additionally, we know from theorem \ref{thm:ExistenceUniquenessDetProd} that the mapping $t \mapsto S_{st}^r$ is continuous, which proves the third statement.
To prove the last statement of this lemma, one can show with properties 1., 2.\ and 3.\ the continuity of the mapping $\phi \colon \{(s,t,y) \in [0,T]^2 \times E \colon s \leq t\} \to E$ in a standard way.
\end{proof}

To keep with the notation in \cite{Jacobsen2006}, we define for every $y = (\vec{r},\vec{q},\vec{\rho}) \in E$ and $B \in \cE$ the mappings
\begin{align}
\psi(t,y) &= \sum_{e=1}^N \gl^e_{r_e r_e}(t,(q_e,\rho_e)),\label{eq:LoadDependentPsiNet}\\
\eta(t,y,B) &= \sum_{e=1}^N  \sum_{l=1,l \neq r_e}^{C^e} \frac{\gl^e_{r_e l}(t,(q_e,\rho_e))}{\psi(t,y)} \epsilon_{(r_1,\dots,r_{e-1},l,r_{e+1},\dots,r_N,\vec{q},\vec{\rho})}(B),\label{eq:LoadDependentEtaNet}
\end{align}
where \[\gl^e_{ij} \colon [0,T] \times \R_{\geq 0}\times L^1((a^e,b^e)) \to \R_{\geq 0},\] $i,j = 1,\dots,C^e$ are given rate functions, satisfying $\gl^e_{ii} = \sum_{j = 1, j\neq i}^{C^e} \gl^e_{ij}$ and $\epsilon_x$ the Dirac measure with unit mass in $x$.
Both functions $\psi$ and $\eta$ model the information about the distribution between the machine failures and the distribution of the corresponding capacity value, see theorem \ref{thm:DistributionMPPLoadDependent} later. We assume a uniform upper bound $\overline{\gl} = \sum_{e=1}^N \overline{\gl^e}$ on the rate functions, i.e.
\begin{align}
\sup\{\gl_{ij}^e(t,q,\rho) \colon i,j=1,\dots,C^e,\; t \in [0,T],\; (q,\rho) \in D\} \leq \overline{\gl^e}.\label{eq:AssumptionRates}
\end{align}
We are  able to state a so-called thinning algorithm for the MLDPNM, where we adopt some ideas of \cite{Lemaire2017}.
Let $(\xi_i,i \in \N)$ be a sequence of independent identically distributed (iid) exponentially distributed random variables with mean $(\overline{\lambda})^{-1}$ on some probability space \OAP{}, and let $(U_i, i \in \N)$ be a sequence of iid uniformly distributed random variables on $[0,1]$ on the same probability space and independent of $(\xi_i,i \in \N)$. If $T_n = t_n \in [0,T)$ and $Y_n = y_n \in E$, then algorithm \ref{alg:thinning_one_queue_proc} produces the next time of a machine failure $T_{n+1}$ and a corresponding value of the whole system $Y_{n+1}$.
\begin{algorithm}[H]
\caption{Thinning algorithm}
\label{alg:thinning_one_queue_proc}
\begin{algorithmic}%[1] % The number tells where the line numbering should start
\STATE $i = 1$
\STATE $s_i = t_n+\xi_i$
\WHILE {$U_i > \psi(s_i, \phi_{t_n s_i}(y_n)) \cdot (\overline{\lambda})^{-1}$}
	\STATE $s_{i+1} = s_i + \xi_i$
	\STATE $i = i+1$
\ENDWHILE
\STATE $T_{n+1} = s_i$
\STATE Generate $Y_{n+1} \sim \eta(s_i,\phi_{t_n s_i}(y_n),\cdot)$
\end{algorithmic}
\end{algorithm}
Algorithm \ref{alg:thinning_one_queue_proc} can be interpreted as follows: starting from $Y_n$ at time $T_n$, we simulate an exponentially distributed random variable $\xi_1$ with mean $(\overline{\lambda})^{-1}$ and solve the deterministic system until the time $s_1 = t_n+\xi_1(\go)$, see figure \ref{fig:ThinningAlgo}. Then, we decide by an acceptance-rejection method whether we switch or keep the value $Y_n$ while the acceptance-rejection method accepts with probability $\psi(s_1,\phi_{t_n s_1}(y_n))$ and rejects with probability $1-\psi(s_1,\phi_{t_n s_1}(y_n))$. In figure \ref{fig:ThinningAlgo}, the time $s_1$ is rejected and $T_{n+1} = s_2$ is accepted. If a time is accepted, the new state $Y_{n+1}$ is simulated according to the probability distribution $\eta(s_i,\phi_{t_n s_i}(y_n),\cdot)$.
\begin{figure}[h]
\begin{picture}(250,90)(0,-10)
% Axes
% x-Axes
\put(50,0){\vector(1,0){250}}
\put(60,-4){\line(0,1){8}}
\put(60,-12){\makebox(5,5){$T_{n-1}$}}
\put(140,-4){\line(0,1){8}}
\put(140,-12){\makebox(5,5){$T_{n}$}}
\put(160,0){\circle{5}}
\put(157.5,-3.2){\makebox(5,5){$\bigtimes$}}
\put(160,-12){\makebox(5,5){$s_1$}}
%\put(190,0){\circle{5}}
%\put(187,-3.25){\makebox(5,5){X}}
%\put(205,0){\circle{5}}
\put(205,0){\circle{5}}
\put(205,-4){\line(0,1){8}}
\put(205,-12){\makebox(5,5){$T_{n+1}$}}
\put(300,-12){\makebox(5,5){$t$}}

% y-Axes
\put(50,0){\vector(0,1){75}}
\put(50,80){\makebox(5,5){$\mu^e(t)$}}

% Markers
\put(60,20){\circle*{5}}
\put(60,26){\makebox(5,5){$Y_{n-1}$}}
\put(60,20){\line(1,0){80}}
\put(140,60){\circle*{5}}
\put(140,66){\makebox(5,5){$Y_{n}$}}
\put(140,60){\line(1,0){65}}
\put(205,40){\circle*{5}}
\put(205,46){\makebox(5,5){$Y_{n+1}$}}
%\put(220,40){\line(1,0){80}}

\end{picture}
\caption{Thinning algorithm example}
\label{fig:ThinningAlgo}
\end{figure}
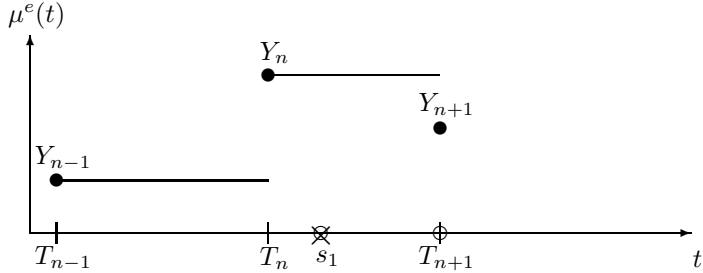

The following theorem \ref{thm:DistributionMPPLoadDependent} contains the information about the probability distribution of the times of capacity drops and the corresponding state values generated by algorithm \ref{alg:thinning_one_queue_proc}. We need this information to prove the existence and property three of a MLDPNM later on. 
Within the proof of theorem \ref{thm:DistributionMPPLoadDependent}, we need the following useful tool:
let, for some $[a,b] \subset \R$, the function $f\colon [a,b]\to \R$ be continuous and $x_0 \in [a,b]$; then, it holds that
\begin{align}
\int_{x_0}^b f(x_1) \int_{x_1}^b f(x_2) \cdots \int_{x_{m-1}}^bf(x_m) dx_m \cdots dx_1 = \frac{1}{m!}\left( \int_{x_0}^b f(z) dz\right)^m \label{lem:IterativeIntegral}
\end{align}
for every $m \in \N$. The proof is an application of the integration by parts formula.
%\begin{proof}
%Let $k \in \N$ and $x \in [a,b]$. We first calculate with partial integration
%\begin{align*}
%&\;\int_x^b f(y) \left(\int_y^b f(z)dz\right)^k dy \\
%= &\; \left[- \left(\int_y^b f(z) dz\right)^{k+1}\right]_x^b- \int_x^b \left(\int_y^b f(z)dz\right) \; k \left(\int_y^bf(z)dz\right)^{k-1}f(y) dy \\
%= &\; \left(\int_x^b f(z) dz\right)^{k+1}-k \int_x^b f(y) \left(\int_y^bf(z)dz\right)^{k} dy,
%\end{align*}
%which is equivalent to
%\begin{align*}
%\int_x^b f(y) \left(\int_y^bf(z)dz\right)^{k} dy = \frac{1}{k+1}\left(\int_x^b f(z) dz\right)^{k+1}.
%\end{align*}
%Using the previous formula, we can calculate
%\begin{align*}
%\int_{x_{m-2}}f(x_{m-1}) \int_{x_{m-1}}^bf(x_m) dx_m dx_{m-1} = \frac{1}{2} \left(\int_{x_{m-2}}^b f(z) dz\right)^2,
%\end{align*}
%and for every $l = 1,\dots,m$, this yields
%\begin{align*}
%\int_{x_{m-l}} f(x_{m-l+1}) \cdots \int_{x_{m-1}}^bf(x_m) dx_m dx_{m-1} \cdots dx_{m-l+1} = \frac{1}{1 \cdot 2 \cdots l} \left(\int_{x_{m-l}}^bf(z) dz\right)^l.
%\end{align*}
%We conclude the theorem by setting $l = m$.
%\end{proof}

%Now, we can state and prove theorem \ref{thm:DistributionMPPLoadDependent}, which tells us the probability distribution of the sequence of random variables constructed with the thinning algorithm \ref{alg:thinning_one_queue_proc}.
\begin{theorem} \label{thm:DistributionMPPLoadDependent}
Assume that the sequence $((T_n,Y_n),n \in \N_0)$ is constructed with algorithm \ref{alg:thinning_one_queue_proc} and $T_0 = t_0, Y_0 = y_0$ for some $t_0 \in [0,t)$ and $y_0 \in E$. Then, for every $n \in \N$ and $t_0 < t_1<\dots t_n \leq t$, $y_1,\dots, y_n  \in E$, it holds that
\begin{align}
&P(T_{n+1} \leq t|T_n = t_n,Y_n = y_n,\dots, T_0 = t_0,Y_0 = y_0) \notag \\
= \;&1-e^{-\int_{t_n}^t \psi(\tau,\phi_{t_n \tau}(y_n)) d\tau}, \label{eq:ProbDistTn}\\[1ex]
&P(Y_{n+1} \in B|T_{n+1} = t,T_n = t_n,Y_n = y_n,\dots, T_0 = t_0,Y_0 = y_0)\notag \\ 
= \;&\eta(t,\phi_{t_n t}(y_n),B)\label{eq:ProbDistYn}
\end{align}
for every $B \in \cE$.
%\begin{enumerate}
%\item
%$\displaystyle
%P(T_{n+1} \leq t|T_n = t_n,Y_n = y_n,\dots, T_0 = t_0,Y_0 = y_0) = 1-e^{-\int_{t_n}^t \psi(\tau,\phi_{t_n \tau}(y_n)) d\tau}
%$ and
%\item
%$\displaystyle
%P(Y_{n+1} \in B|T_{n+1} = t,T_n = t_n,Y_n = y_n,\dots, T_0 = t_0,Y_0 = y_0) = \eta(t,\phi_{t_n t}(y_n),B)
%$ for every $B \in \cE$.
%\end{enumerate}
\end{theorem}
\begin{proof}
The second statement \eqref{eq:ProbDistYn} of the theorem follows directly from the last line of  algorithm \ref{alg:thinning_one_queue_proc}. We now only have to show the first statement, where we follow the ideas in \cite{YChen2006} for the inhomogeneous Poisson process.
Since the algorithm used to generate $T_{n+1}$ and $Y_{n+1}$ only needs the values of $T_n$ and $Y_n$, we have
\begin{align*}
P(T_{n+1} \leq t|T_n = t_n,Y_n = y_n,\dots, T_0 = t_0,Y_0 = y_0) = P(T_{n+1} \leq t|T_n = t_n,Y_n = y_n).
\end{align*}
By defining \[\tilde{\psi}(s)= \frac{\psi(s,\phi_{t_n s}(y_n))}{\overline{\lambda}},\quad s_i = t_n+\sum_{l=1}^i \xi_l\quad \text{ and }\quad  M = \inf\{m \in \N \colon U_m \leq \tilde{\psi}(s_m) \},\] we can write $T_{n+1} = s_M$.
Thus,
\[P(T_{n+1} \leq t|T_n = t_n,Y_n = y_n) = \sum_{m=1}^\infty P(s_m \leq t, M = m|T_n = t_n,Y_n = y_n) \]
and
\begin{align*}
&P(s_m \leq t, M = m|T_n = t_n,Y_n = y_n) \\[1ex]
= \;&P(s_m \leq t, U_1 > \tilde{\psi}(s_1),\dots,U_{m-1} > \tilde{\psi}(s_{m-1}),U_m \leq \tilde{\psi}(s_m)|T_n = t_n,Y_n = y_n) \\
= \; &\int_{\R^m} \Ind_{[0,t-t_n]}(h_m(\vec{x})) \tilde{\psi}(t_n+h_m(\vec{x})) \prod_{i=1}^{m-1} (1-\tilde{\psi}(t_n+h_i(\vec{x}))) \overline{\lambda}^m e^{-\overline{\lambda} h_m(\vec{x})} \Ind_{\R_{\geq 0}^m}(\vec{x}) d\vec{x},
\end{align*}
where we used that $\xi_1,\dots,\xi_m$ are iid exponentially distributed and $h_i(\vec{x}) = \sum_{l = 1}^i x_l$. We use the transformation $\tau_k = h_k(\vec{x})$, which implies $x_k = \tau_{k}-\tau_{k-1}$ with $\tau_0 = 0$. The last integral reads with integration by substitution, see, e.g., \cite{CasellaBerger2002},
\begin{align*}
&\int_{\R^m} \Ind_{[0,t-t_n]}(\tau_m) \tilde{\psi}(t_n+\tau_m) \prod_{i=1}^{m-1} (1-\tilde{\psi}(t_n+\tau_i)) \overline{\lambda}^m e^{-\overline{\lambda} \tau_m} \prod_{i=1}^m \Ind_{\R_{\geq 0}}(\tau_{k}-\tau_{k-1}) d\vec{\tau}\\
=&\;\int_{[0,t-t_n]^m} \tilde{\psi}(t_n+\tau_m) \prod_{i=1}^{m-1} (1-\tilde{\psi}(t_n+\tau_i)) \overline{\lambda}^m e^{-\overline{\lambda} \tau_m} \prod_{i=1}^m \Ind_{\R_{\geq 0}}(\tau_{k}-\tau_{k-1}) d\vec{\tau}\\
=&\;\int_0^{t-t_n}(1-\tilde{\psi}(t_n+\tau_1))\int_{\tau_1}^{t-t_n}(1-\tilde{\psi}(t_n+\tau_2)) 
\\ &\quad \cdots \int_{\tau_{m-1}}^{t-t_n} \tilde{\psi}(t_n+\tau_m)\overline{\lambda}^m e^{-\overline{\lambda} \tau_m} d\tau_m \cdots d\tau_1.
\end{align*}
We define for $k =1,\dots,m$
\begin{align*}
a_k =& \int_0^{t-t_n}(1-\tilde{\psi}(t_n+\tau_1))\int_{\tau_1}^{t-t_n}(1-\tilde{\psi}(t_n+\tau_2)) 
\\ &\quad \cdots \int_{\tau_{k-1}}^{t-t_n}(1-\tilde{\psi}(t_n+\tau_k))\overline{\lambda}^k e^{-\overline{\lambda} \tau_k} d\tau_k \cdots d\tau_1
\end{align*}
and  for $k = 0$, we set $a_0 = 1$. Rewriting
\begin{align*}
&\int_{\tau_{m-1}}^{t-t_n} \tilde{\psi}(t_n+\tau_m)\overline{\lambda}^m e^{-\overline{\lambda} \tau_m} d\tau_m\\
=\;&-\int_{\tau_{m-1}}^{t-t_n} (1-\tilde{\psi}(t_n+\tau_m))\overline{\lambda}^m e^{-\overline{\lambda} \tau_m} d\tau_m + \overline{\lambda}^{m-1}(e^{\overline{\lambda}\tau_{m-1}}-e^{-\overline{\lambda}(t-t_n)})
\end{align*}
implies
\begin{align*}
&P(s_m \leq t, M = m|T_n = t_n,Y_n = y_n) \\[1ex]
=\;& -a_m+a_{m-1}\\
&\;-\overline{\lambda}^{m-1}e^{-\overline{\lambda}(t-t_n)}
\int_0^{t-t_n}(1-\tilde{\psi}(t_n+\tau_1))\cdots \int_{\tau_{m-2}}^{t-t_n}(1-\tilde{\psi}(t_n+\tau_{m-1})) d\tau_{m-1} \cdots d\tau_1\\
=\;& a_{m-1}-a_m-e^{-\overline{\lambda}(t-t_n)} \frac{1}{(m-1)!}\left(\overline{\lambda} \int_0^{t-t_n}(1-\tilde{\psi}(t_n+\tau))d\tau\right)^{m-1}
\end{align*}
with \eqref{lem:IterativeIntegral}.
Summing over $m$, using the telescopic sum and the definition of the exponential function, leads to
\begin{align*}
P(T_{n+1} \leq t|T_n = t_n,Y_n = y_n) &= \sum_{m=1}^\infty P(s_m \leq t, M = m|T_n = t_n,Y_n = y_n)\\
&= a_0-\lim_{m\to \infty} a_m-e^{-\overline{\lambda}(t-t_n)}e^{ \int_0^{t-t_n}\overline{\lambda} (1-\tilde{\psi}(t_n+\tau))d\tau}\\
&= 1-0-e^{-\int_{t_n}^t\psi(\tau,\phi_{t_n \tau}(y_n)d\tau)},
\end{align*}
where we used $a_0 = 1$ and $\lim_{m\to \infty} a_m = 0$. This completes the proof.
\end{proof}

The following, with our notation and situation adapted theorem \ref{thm:PDMPJacobsen2006}, taken from \cite{Jacobsen2006}, provides the main tool to show the existence of a MLDPNM.
\begin{theorem}\label{thm:PDMPJacobsen2006}
Let 
$\phi \colon \{(s,t,y) \in [0,T]^2 \times E \colon s \leq t\} \to E$ be measurable and satisfy statements 1.-3.\ of lemma \ref{lem:PropPhi}. Assume that $(t,y) \mapsto \psi(t,y)$ is $\cB([0,T]) \otimes \cE$-measurable and that, for every $(t,y) \in [0,T] \times E$, it holds that $\int_t^{t+h} \psi(\tau, \phi_{t\tau}(y)) d\tau < \infty$ for some $h = h(t,y)>0$ sufficiently small. Let for every $t \in [0,T]$ the mapping $(y,B) \mapsto \eta(t,y,B)$ be a Markovian kernel on $(E, \cE)$, and let the function $(t,y) \mapsto \eta(t,y,B)$ $\cB([0,T]) \otimes \cE$-measurable for every $B \in \cE$ with $\eta(t,y,\{y\}) = 0$ for every $(t,y) \in [0,T] \times E$.
If $\mu = \sum_{n \in \N_0, T_n<\infty} \epsilon_{(T_n,Y_n)}$ is a stable random counting measure determined by the marked point process $((T_n,Y_n),n \in \N)$ on \OAP{} satisfying statements 1.\ and 2.\ of theorem \ref{thm:DistributionMPPLoadDependent}, then $(X(t), t \in [0,T])$ defined by
\begin{align*}
X(t) = \phi_{T_{N(t)}t}(Y_{N(t)}) \text{ with }X(0) = Y_0 = x_0 \text{ and } N(t) = \inf\{n \in \N \colon T_n \leq t < T_{n+1}\}
\end{align*}
is a piecewise deterministic Markov process on \OAP{} with respect to the natural filtration. 
\end{theorem}

The next theorem \ref{thm:ExistenceLoadDependentNetwork} addresses the question of the existence of an MLDPNM.
\begin{theorem}\label{thm:ExistenceLoadDependentNetwork}
Let a production network be given. Additionally, we have $$C^1,\dots,C^N \in \N,\quad \mu \colon \bigtimes_{e=1}^N\{1,\dots,C^e\} \to \R_{\geq 0}^N$$ and continuous transition rate functions \[\gl^e_{ij} \colon [0,T] \times \R_{\geq 0}\times L^1((a^e,b^e)) \to \R_{\geq 0},\]
 $i,j = 1,\dots,C^e$, $e \in \cC$ satisfying $\gl^e_{ii} = \sum_{j = 1, j\neq i}^{C^e} \gl^e_{ij}$ and \eqref{eq:AssumptionRates}. Then, there exists for all initial data $(r_0,q_0,\rho_0) \in E$ an MLDPNM.
\end{theorem}
\begin{proof}
We use theorem \ref{thm:PDMPJacobsen2006} to show the existence. Let $\phi$ be as in \eqref{eq:LoadDepPhiNet}; then, from lemma \ref{lem:PropPhi}, we deduce the assumptions on $\phi$ that we need for theorem \ref{thm:PDMPJacobsen2006}. By defining $\psi$ and $\eta$ via \eqref{eq:LoadDependentPsiNet} and \eqref{eq:LoadDependentEtaNet}, we can use assumption \eqref{eq:AssumptionRates} and conclude the integrability
\[\int_t^{t+h} \psi(\tau,\phi_{t,\tau}(y))d\tau \leq \overline{\lambda} h.\] 
Since the rate functions $\gl^e_{ij}$ are measurable and since $i,j \in \{1,\dots,C^e\}$ is a finite set, the mappings $(t,y) \mapsto \psi(t,y)$ and $(t,y) \mapsto \eta(t,y,B)$ are measurable. We deduce $\eta(t,y,\{y\}) = 0$ from $\gl^e_{ii} = \sum_{j = 1, j\neq i}^C \gl^e_{ij}$, and with $\eta(t,y,E) = 1$, we easily see that $(y,B) \mapsto \eta(t,y,B)$ defines a Markovian kernel.\\
Let $((T_n,Y_n),n \in \N)$ be constructed with algorithm \ref{alg:thinning_one_queue_proc}. The corresponding counting measure is stable since, by thinning, the number of jumps is less than the number of jumps of the Poisson process with rate $\overline{\gl}$, which is stable; see \cite{Jacobsen2006}. Statements 1.\ and 2.\ of theorem \ref{thm:DistributionMPPLoadDependent} are satisfied by construction, and $(X(t), t \in [0,T])$ with 
\begin{align*}
X(t) = \phi_{T_{N(t)}t}(Y_{N(t)}) \text{ with }X(0) = Y_0 = x_0, \text{ and } N(t) = \inf\{n \in \N \colon T_n \leq t < T_{n+1}\},
\end{align*}
is a piecewise deterministic Markov process with respect to the natural filtration by theorem \ref{thm:PDMPJacobsen2006}. Thus, $(X(t), t \in [0,T])$ satisfies conditions 1., 2.\ and 4.\ of definition \ref{def:MLDPNM}. It remains to show condition 3. The Markov property and the fact that the state space is a Polish space allows us to construct a canonical coordinate process $(\tilde{X}(t), t \in [0,T])$ on some probability space $(\tilde{\Omega},\tilde{\cA},\tilde{P})$, which starts in $y \in E$ having the same finite dimensional distributions like $(X(t), t \in [0,T])$. 
Let $\tilde{e} \in \cC$ and $j \in \{1,\dots,C^{\tilde{e}}\}$; then, 
\begin{align*}
P(r^{\tilde{e}}(t+\Delta t) = j|X(t) = y) = \tilde{P}(\tilde{r}^{\tilde{e}}(\Delta t) = j|\tilde{X}(0) = y)
\end{align*}
for every $y \in E$ and $(\tilde{X}(t), t \in [0,T])$, the canonical coordinate process starting in $y = (r,q,\rho)$ with 
\begin{align*}
\tilde{\psi}(\tau,y) &= \psi(t+\tau,y) \text{ and }\\[1ex]
\tilde{\eta}(\tau,y,B) &= \eta(t+\tau,y,B).
\end{align*}

We compute
\begin{align*}
\tilde{P}(\tilde{r}^{\tilde{e}}(\Delta t) = j|\tilde{X}(0) = y) &= \tilde{P}(\tilde{r}^{\tilde{e}}(\Delta t) = j, \tilde{T}_1 \leq \Delta t | \tilde{X}(0) = y)\\[1ex]
+&\tilde{P}(\tilde{r}^{\tilde{e}}(\Delta t) = j | \tilde{T}_1 > \Delta t ,\tilde{X}(0) = y)P(\tilde{T}_1 > \Delta t | \tilde{X}(0) = y)
\end{align*}
with $\tilde{T}_1$ the first jump time of the process $(\tilde{X}(t), t \in [0,T])$. The function 
\[f(\tau) = \sum_{e = 1}^N \gl_{r_e r_e}^e(\tau,S_{t \tau}^{r}(q,\rho))\]
is continuous with the assumptions, and lemma \ref{lem:PropPhi} can be applied such that
\begin{align*}
P(\tilde{T}_1 \leq \Delta t| \tilde{X}(0) = y) &= 1-e^{-\int_t^{t+\Delta t} f(\tau) d\tau}\\[1ex]
	&= \Delta t f(t) + \lano(\Delta t).
\end{align*}
Therefore, it holds that
\begin{align*}
&\;\tilde{P}(\tilde{r}^{\tilde{e}}(\Delta t) = j | \tilde{T}_1 > \Delta t ,\tilde{X}(0) = y)P(\tilde{T}_1 > \Delta t | \tilde{X}(0) = y) \\[1ex]
=&\; (1-\Delta t f(t)) \Ind_j(r) +\lano(\Delta t) \\[1ex]
=&\; (1-\Delta t \psi(t,y)) \Ind_j(r) + \lano(\Delta t)
\end{align*}
as $ \Delta t \to 0$.
We set 
\begin{align*}
B = \bigtimes_{i = 1}^{\tilde{e}-1} \{1,\dots,C^i\} \times \{ j \} \times \bigtimes_{i = \tilde{e}+1}^{N} \{1,\dots,C^i\} \times D
\end{align*}
and compute
\begin{align*}
P(\tilde{r}(\Delta t) = j,\tilde{T}_1 \leq \Delta t |\tilde{X}(0) = y)
= P(\tilde{Y}_1 \in B,\tilde{T}_1 \leq \Delta t| \tilde{X}(0) = y)+\lano(\Delta t).
\end{align*}
This is valid because the probability of strictly more than one jump in $[0,\Delta t]$ is of order $\lano(\Delta t)$. Due to the continuity of $f$ and $\gl$, we can again calculate
\begin{align*}
&\; P(\tilde{Y}_1 \in B,\tilde{T}_1 \leq \Delta t| \tilde{X}(0) = y)\\
=&\; \int_0^{\Delta t} \tilde{\eta}(s,S^r_{t t+s}(q,\rho), B) \frac{d}{ds} (1-e^{-\int_0^s f(t+\tau) d\tau}) ds\\
=&\; \int_0^{\Delta t} \eta(t+s,S^r_{t t+s}(q,\rho), B) \psi(t,y) e^{-\int_0^s f(t+\tau) d\tau}) ds \\[1ex]
=&\; \eta(t,y,B) \psi(t,y) \Delta t + \lano(\Delta t).
\end{align*}
In the following, we compute $\eta(t,y,B) \psi(t,y)$ in several steps. For fixed $e \in \{1,\dots,N\}$ and $l \in \{1,\dots,C^e\}$ with $l\neq r_e$, we obtain 
\begin{align*}
\epsilon_{(r_1,\dots,r_{e-1},l,r_{e+1},\dots,r_N,\vec{q},\vec{\rho})}(B) = \epsilon_l(\{j\}) \Ind_e(\tilde{e}) + \epsilon_{r_{\tilde{e}}}(\{j\}) (1-\Ind_e(\tilde{e}))
\end{align*}
from the structure of the set $B$. Hence,
\begin{align*}
&\; \sum_{\substack{l = 1\\l \neq r_e}}^{C^e} \gl^e_{r_e l}(t,(q_e,\rho_e))\epsilon_{(r_1,\dots,r_{e-1},l,r_{e+1},\dots,r_N,\vec{q},\vec{\rho})}(B)\\
=&\; \Ind_e(\tilde{e})  (1-\Ind_{r_e}(j)) \gl^e_{r_e j}+(1-\Ind_e(\tilde{e}))+\epsilon_{r_{\tilde{e}}}(\{j\})(1-\Ind_e(\tilde{e})) \gl_{r_e r_e}^e
\end{align*}
from $\sum_{l = 1, l \neq r_e}^{C^e} \gl^e_{r_e l} = \gl_{r_e r_e}^e$ by assumption. Summing over $e$ leads to
\begin{align*}
&\; \eta(t,y,B) \psi(t,y)\\[1ex]
=&\; \sum_{e=1}^N \sum_{\substack{l = 1\\ l \neq r_e}}^{C^e} \gl^e_{r_e l}(t,(q_e,\rho_e))\epsilon_{(r_1,\dots,r_{e-1},l,r_{e+1},\dots,r_N,\vec{q},\vec{\rho})}(B)\\
=&\;  (1-\Ind_{r_{\tilde{e}}}(j)) \gl^{\tilde{e}}_{r_{\tilde{e}} j} + \Ind_{r_{\tilde{e}}}(j)\sum_{\substack{e = 1\\ e\neq \tilde{e}}}^N \gl_{r_e r_e}^e\\
=&\;  (1-\Ind_{r_{\tilde{e}}}(j)) \gl^{\tilde{e}}_{r_{\tilde{e}} j} + \Ind_{r_{\tilde{e}}}(j)(\psi(t,y)-\gl^{\tilde{e}}_{r_{\tilde{e}} r_{\tilde{e}}}).
\end{align*}
Consequently, we end up with
\begin{align*}
\tilde{P}(\tilde{r}^{\tilde{e}}(\Delta t) = j|\tilde{X}(0) = y) =&\; (1-\Ind_{r_{\tilde{e}}}(j)) \Delta t \gl^{\tilde{e}}_{r_{\tilde{e}} j} \\[1ex]
&\;+ \Ind_{r_{\tilde{e}}}(j)\Delta t (\psi(t,y)-\gl^{\tilde{e}}_{r_{\tilde{e}}  r_{\tilde{e}}}) \\[1ex]
&\;+\Ind_{r_{\tilde{e}}}(j)(1-\Delta t \psi(t,y)) + \lano(\Delta t)\\[1ex]
=&\; (1-\Ind_{r_{\tilde{e}}}(j)) \Delta t \gl^{\tilde{e}}_{r_{\tilde{e}} j} + \Ind_{r_{\tilde{e}}}(j)(1-\Delta t \gl^{\tilde{e}}_{r_{\tilde{e}}  r_{\tilde{e}}})  + \lano(\Delta t)
\end{align*}
and finish the proof.
\end{proof}

Before we numerically analyze the MLDPNM in the subsequent section, we give some comments on this model. If we assume $\rho^e_0 \in \operatorname{BV}((a^e,b^e))$, we can use property 6.\ of theorem \ref{thm:ExistenceUniquenessDetProd} and conclude from algorithm \ref{alg:thinning_one_queue_proc} with \eqref{eq:LoadDependentPsiNet}-\eqref{eq:LoadDependentEtaNet} that $\rho^e(t) \in \operatorname{BV}((a^e,b^e))$ $P$-a.s. Hence, there exists a set $\cN \in \cA$ such that \[t\mapsto S_{T_k(\go) t}^{\mu(r(T_k(\go),\go))}(q(T_k(\go),\go),\rho(T_k(\go),\go))\]
is a network solution on $[T_k(\go),T_{k+1}(\go))$ for all $k$ and $\go \in \gO \setminus \cN$.
We claimed the uniform upper bound on the transition rates in equation \eqref{eq:AssumptionRates}, which is quite restrictive. From the deterministic production network model, we deduce formally the subsequent bounds on the solution:
\begin{enumerate}
\item $q^e(t) \leq q_e^{\text{max}}$ with 
\begin{align*}
q_e^{\text{max}} &=
\begin{cases}
q_0^e + \int_0^T G_{\text{in}}^{s(e)}(t) dt &\text{ if } s(e) \in V_{\text{in}},\\[1ex]
q_0^e + \sum_{\tilde{e} \in \delta_{s(e)}^-} \mu_{\tilde{e}}^{\text{max}} \int_0^T A^{s(e),e}(t) dt &\text{ otherwise,} 
\end{cases} \\
\mu_e^{\text{max}} &= \max_{r = 1,\dots,C^e} \mu^e(r)
\end{align*}
\item if $v^e \rho^e_0 (x)\leq \mu^e$, we have $v^e \rho^e(x,t) \leq \mu^e$, and in general, 
$$\rho^e(x,t) \leq \max \Big\{\|\rho_0^e\|_{L^\infty((a,b))}, \frac{\mu_e^{\text{max}}}{v^e}\Big\}.$$
\end{enumerate}
We obtain these bounds $P$-a.s.\ for the MLDPNM, which allows us to relax the uniform bound on the transition rates as follows: there exist $0<\overline{\gl^e} <\infty$, such that
\begin{align}
 \sup\Big\{\gl^e_{ij}(t,q,\rho) \colon &i,j =1,\dots,C^e,\;t \in [0,T],\; (q,\rho) \in \R_{\geq 0}\times L^1((a^e,b^e)); \notag \\ &q \leq q_e^{\text{max}}, \; 0\leq \rho \leq \max\Big\{\|\rho_0^e\|_{L^\infty},\frac{\mu_e^{\text{max}}}{v^e}\Big\} \Big\} \leq \overline{\gl^e} \label{eq:AssumptionRatesNetwork2}
\end{align}
for every $e = 1,\dots, N$.

%%%%%%%%%%%%%%%%%%%%%%%%%%%%%%%%%%%%%%%%%%%%%%%%%%%%%%%%%%%%%%%%%%%%%%%%%%%%%%%%%%%%%%%%%%%%%%%%%%%%%%%%%%%%%%%%%%%%%%%%%%%%%%%%%%%%%%%%%%%%%%%%%%%%%%%%%%%%%%%%%%%
%%%%%%%%%%%%%%%%%%%%%%%%%%%%%%%%%%%%%%%%%%%%%%%%%%%%%%%%%%%%%%%%%%%%%%%%%%%%%%%%%%%%%%%%%%%%%%%%%%%%%%%%%%%%%%%%%%%%%%%%%%%%%%%%%%%%%%%%%%%%%%%%%%%%%%%%%%%%%%%%%%%
\section{Numerical Treatment and Computational Results}\label{sec:NumTreat}
As proposed in \cite{GoettlichKnapp2017,GoettlichMartinSickenberger}, we use the fact that a PDE-ODE system can be solved deterministically between the random switching times in the sense of a piecewise deterministic process~\cite{Davis1984}. 

Therefore, we have to address two issues to state a suitable numerical approximation for a MLDPNM: first, the generation of the marked point process $((T_n,Y_n),n \in \N_0)$ and second, the approximation of the deterministic PDE-ODE system between the random times.

The generation of the marked point process $((T_n,Y_n),n \in \N_0)$ is straightforward with algorithm \ref{alg:thinning_one_queue_proc} and the numerical approximation of the deterministic evolution between the jump times can be performed with the scheme presented in \cite{GoettlichMartinSickenberger}.
Therein, the queue lengths are approximated with a forward Euler method and the densities are approximated with a left-sided upwind scheme and coupled by the boundary conditions to an approximation to the network solution.

In the following, we introduce reasonable choices for the load-dependent rate functions $\gl_{ij}^e$ and, in a second step, we analyze the impact of the load-dependency of the production network model on performance measures

\subsection{Utilization Ratio and Work-In-Progress} \label{subsec:UR_RWIP}
At the beginning, the load-dependent model was motivated by a load-dependent probability of machine failures. 
Mathematically, the load-dependent probabilities are included by load-dependent transition rate functions, which we specify in the following.
Let $X$ be a MLDPNM and we assume a network consisting of a single queue-processor unit for simplicity here. Then, we define the Utilization Ratio $\operatorname{UR}$ by 
\begin{align*}
\operatorname{UR} \colon E &\to [0,1] \\
X(t) &\mapsto \frac{1}{\mu^{\text{max}}(b-a)} \int_a^b \min\{\mu(r(t)),v\rho(x,t)\}dx,
\end{align*}
which corresponds to the average production flow compared to the maximal production flow of the machine at time $t$. A meaningful assumption is the dependence of capacity drops given by the rates $\gl$ on the Utilization Ratio of the machine. Specifically, if there is no production ($\operatorname{UR}= 0$), there is no reason for a breakdown and only machine care has to be done with a rate $\gl^{\text{down,min}} \geq 0$. If the production is at the capacity limit ($\operatorname{UR} = 1$), we expect a higher probability of a machine failure caused by high abrasion.
In the case of two possible capacities, i.e., $C = 2$ and $\mu(1) = \mu^{\text{min}}, \mu(2) = \mu^{\text{max}}$, a simple relationship is given by the linear relation
\begin{align}
\gl_{21}(q(t),\rho(t)) = \gl^{\text{down,min}}+(\gl^{\text{down,max}}-\gl^{\text{down,min}})\operatorname{UR}(2,q(t),\rho(t)), \label{eq:LoadDepRate1}
\end{align}
where $\gl^{\text{down,max}} > 0$ is the rate of a machine failure in the case of a production at maximal capacity. A simple calculation shows 
\begin{align*}
|\gl_{21}(q,\rho)-\gl_{21}(\tilde{q},\tilde{\rho})| \leq \frac{(\gl^{\text{down,max}}-\gl^{\text{down,min}})v}{\mu^{\text{max}}(b-a)} \|\rho-\tilde{\rho}\|_{L^1((a,b))}
\end{align*}
and implies the Lipschitz continuity of this rate function.\\
It remains to define a dependence of the production on the rate function in the case of a repair. One possible dependence can be given by the Ratio of Work In Progress and the maximal amount of goods in the machine $\operatorname{RWIP}$, which we define by 
\begin{align*}
\operatorname{RWIP} \colon E &\to [0,1]\\
X(t) &\mapsto \frac{v}{\mu^{\text{max}}(b-a)}\int_a^b \rho(x,t)dx.
\end{align*}
In the case of a low work in progress, the machine can easily be repaired: there are fewer goods that have to be removed for repairs. With a high work in progress, this takes more time. Let $\gl^{\text{rep,max}}>0$ be the repair rate in the case of an empty system, and let $\gl^{\text{rep,min}}\geq 0$ be the repair rate for a full machine such that $\gl^{\text{rep,max}} \geq \gl^{\text{rep,min}}$. Again, the simplest relation is the linear relation 
\begin{align}
\gl_{12}(q(t),\rho(t)) = \gl^{\text{rep,max}}-(\gl^{\text{rep,max}}-\gl^{\text{rep,min}})\operatorname{RWIP}(1,q(t),\rho(t)).
\label{eq:LoadDepRate2}
\end{align}
This rate function is again Lipschitz continuous with
\begin{align*}
|\gl_{12}(q,\rho)-\gl_{12}(\tilde{q},\tilde{\rho})| \leq \frac{v(\gl^{\text{rep,max}}-\gl^{\text{rep,min}})}{\mu^{\text{max}}(b-a)} \|\rho-\tilde{\rho}\|_{L^1((a,b))}.
\end{align*}
Provided $v\rho_0 \leq \mu^{\text{max}}$, the choice 
\begin{align*}
\overline{\gl} = \max\{\gl^{\text{down}},\gl^{\text{rep,max}}\}
\end{align*}
is a uniform bound on the transition rates and can be used for the thinning algorithm \ref{alg:thinning_one_queue_proc}.
All these ideas can be extended to more capacity states and to general network topologies in a straightforward way.

\subsection{Performance Measures}\label{subsec:PerformanceMeasures}
To evaluate the performance of the production network and its numerical behavior,
performance measures need to be defined. For our purposes, these are
the mean queue-load, the mean outflow and the distribution of the random variables. 
According to \cite{GoettlichKnapp2017,GoettlichMartinSickenberger}, the cumulative sum of all queue-lengths and the cumulative outflow until time $t\geq 0$ are given by 
\begin{align*}
q^{\text{net}}(t) = \sum_{e\in \cC} \int_{t_0}^t q^e(s) ds, \quad G^{\text{net}}_{\text{out}}(t) &= \int_{0}^t \sum_{v \in V_{\text{out}}} \sum_{e \in \delta^-_v} f^e(s,\rho^e(b^e,s)) ds,
\end{align*}
where $V_{\text{out}} = \{v \in \cV \colon \delta^+_v = \emptyset\}$. 
Since the above measures are random variables with unknown distributions, some estimators (e.g.,\ moments) are required. We denote as common by $\overline{Y}$ the mean value estimator of $Y \in \{G^{\text{net}}_{\text{out}}(t),q^{\text{net}}(t)\}$. 
The strong convergence of these estimators follows directly from the law of large numbers, i.e., $G^{\text{net}}_{\text{out}}(t), q^{\text{net}}(t) $ are finite for finite time horizons. 

\subsection{Comparison of Load-dependent with Load-independent Model}
\label{subsec:Comparison}
We study a production network model, which topology is given by the diamond network in figure \ref{fig:DiamondNetworkTopology}. Here, $\alpha_1,\alpha_2 \in [0,1]$ are two distribution parameters, i.e., a percentage of $\alpha_1$ is fed from processor one into queue two ($A^{1,2}(t) = \ga_1$), $1-\alpha_1$ from one to three ($A^{1,3} = 1-\ga_1$) and the same for $\alpha_2$ from processor two to queue five and $1-\ga_2$ to queue four. In our numerical investigations we set the distribution parameters fixed with the values $\ga_1 = \ga_2 = 0.5$.
%\begin{figure}[htb!]
%\centering
%\includegraphics[width=0.9\textwidth]{4_ProductionLoadDependent_DiamondNetwork.eps}
%\caption{Diamond network with seven processors}
%\label{fig:DiamondNetworkTopology} 
%\end{figure}
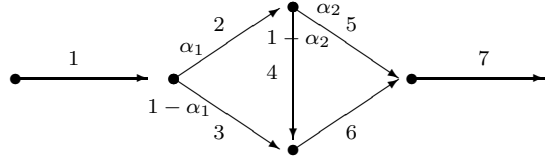
\begin{figure}[H]
\centering
\begin{picture}(285,70)(+19,-45)
% Axes
% x-Axes
\put(60,0){\vector(1,0){50}}
\put(60,0){\circle*{4}}
\put(80,5){\footnotesize{$1$}}

\put(120,0){\vector(3,2){40}}
\put(120,0){\circle*{4}}
\put(135,18){\footnotesize{$2$}}
\put(122,10){\footnotesize{$\alpha_1$}}

\put(120,0){\vector(3,-2){40}}
\put(120,0){\circle*{4}}
\put(135,-23){\footnotesize{$3$}}
\put(110,-13){\footnotesize{$1-\alpha_1$}}

\put(165,27){\vector(0,-1){50}}
\put(165,27){\circle*{4}}
\put(155,0){\footnotesize{$4$}}

\put(165,27){\vector(3,-2){40}}
\put(165,27){\circle*{4}}
\put(185,18){\footnotesize{$5$}}
\put(174,25){\footnotesize{$\alpha_2$}}

\put(165,-27){\vector(3,2){40}}
\put(165,-27){\circle*{4}}
\put(185,-23){\footnotesize{$6$}}
\put(155,13){\footnotesize{$1-\alpha_2$}}

\put(210,0){\vector(1,0){50}}
\put(210,0){\circle*{4}}
\put(235,5){\footnotesize{$7$}}
\end{picture}
\caption{Diamond network with seven processors}
\label{fig:DiamondNetworkTopology} 
\end{figure}

All processors have the same lengths $l^e = b^e-a^e = 1$ with $a^e = 0$ and they share the same processing velocities $v^e = 1$, the same capacity states, i.e., $C^e = 2$ and the same capacity functions $\mu^e = 2\cdot \Ind_2$. We assume that all processors have the same transition rates, where
\begin{alignat*}{2}
\gl^{\text{down,min}} &= (1-\beta)\gl^{\text{down,ref}} ,\quad \gl^{\text{down,max}} &&= 2 \gl^{\text{down,ref}} -\gl^{\text{down,min}} \text{ and }\\
\gl^{\text{rep,min}} &= (1-\beta)\gl^{\text{rep,ref}} ,\quad \gl^{\text{rep,max}} &&= 2 \gl^{\text{rep,ref}} -\gl^{\text{rep,min}}.
\end{alignat*}
The rates $\gl^{\text{down,ref}}>0$ and $\gl^{\text{rep,ref}}>0$ are reference rates and set to $\gl^{\text{down,ref}} = \frac{1}{0.85}$ and $\gl^{\text{rep,ref}} = \frac{1}{0.15}$ in our simulations. The parameter $\beta  \in [0,1]$ allows to study the influence of the load-dependency in the following. More precisely, the choice $\beta = 0$ leads to the load-independent and the choice $\beta \in (0,1]$ to the load-dependent model. We remark that the case $\beta = 0$ is exactly the model presented in \cite{GoettlichMartinSickenberger}.
The network inflow is given by $G_{\text{in}}^{1}(t) = 1.5$, which is lower than the mean capacity in the load-independent case, i.e., the (stationary) availability of a processor is given by $\frac{\gl^{\text{rep,ref}}}{\gl^{\text{down,ref}}+\gl^{\text{rep,ref}}} = 0.85$ and consequently, the (stationary) expected capacity is $1.7$.

For the numerical approximation we use equidistant grids given by a spatial discretization with $\Delta x = \frac{1}{10}$ and a time discretization with $\Delta t = \Delta x$, which satisfies the CFL condition. 
The following results are based on a simulation run with a sample size of $10^4$, and we always start with an intact and empty production.

Figure \ref{fig:MeanCapacity1} shows the expected capacity with respect to the load-dependency scale $\beta$ for processor 1 and 7. The case $\beta = 0$ is well-known from \cite{GoettlichMartinSickenberger} and reproduced here. In fact, from the intact system with a capacity of $\mu = 2$ at time zero, we observe a convergence towards the stationary expected capacity of $1.7$ for processor 1 and 2.
Since for $\beta>0$ the load influences the capacity, we distinguish the description of processor 1 and 7. We start with the first processor:
in the case $\beta = 1$, we see a strong decreasing mean capacity, which implies a stuck in the system. The stationary expected capacities for $\beta = 0.25$ and $\beta = 0.5$ are approximately $1.57$ and $1.35$, respectively. Consequently, the influence of the load-dependency on the mean capacity is quite high.

For processor 7 we observe a similar trend regarding the stationary capacities but they are less affected by the load-dependency. In fact, for $\beta  = 1$, we observe an expected capacity of 2 in the first periods until the first products reach the seventh machine.

The different expected capacities allow the assumption that the expected queue lengths are also affected by the workload-dependency, which is confirmed by figure \ref{fig:MeanQueueLoads1}. By increasing the load-dependency, the expected queue-length increases in this example. The expected queue-length of the first processor increases slowly for $\beta \in \{0,0.25,0.5\}$, moderately for $\beta = 0.75$ and very strongly for $\beta = 1$. Although the last processor 7 shares the same properties of all other processors, the expected queue-length is much smaller compared to the first processor queue-lengths. This is reasonable, since the products get stuck in the production as $\beta$ increases.
\begin{figure}[H]
\begin{subfigure}[c]{0.5\textwidth}
\centering
\includegraphics[width=0.9\textwidth]{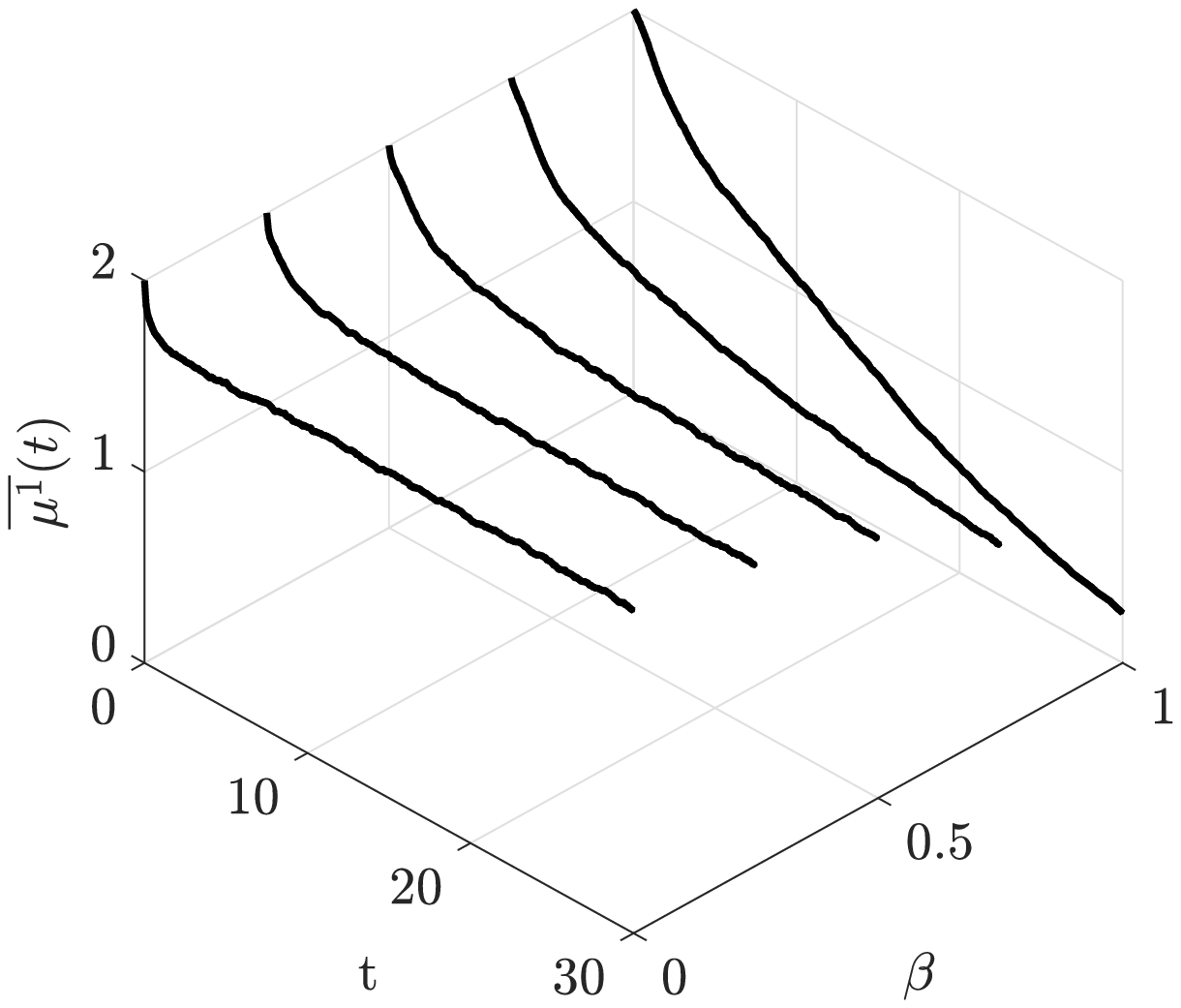}
\subcaption{Processor 1}
\end{subfigure}
\begin{subfigure}[c]{0.5\textwidth}
\centering
\includegraphics[width=0.9\textwidth]{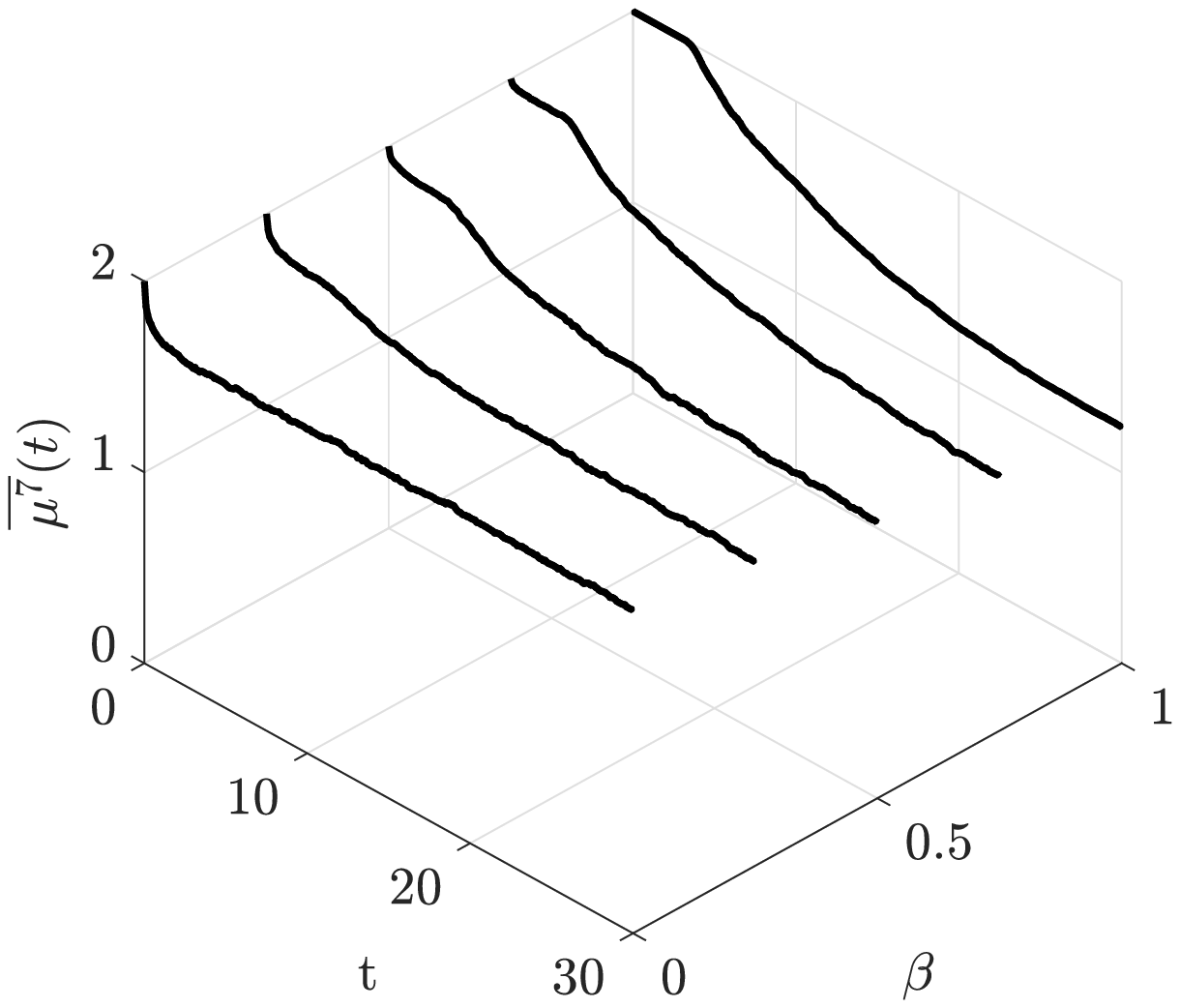}
\subcaption{Processor 7}
\end{subfigure}
\caption{Mean capacity}
\label{fig:MeanCapacity1}
\end{figure}
\begin{figure}[H]
\begin{subfigure}[c]{0.5\textwidth}
\centering
\includegraphics[width=0.9\textwidth]{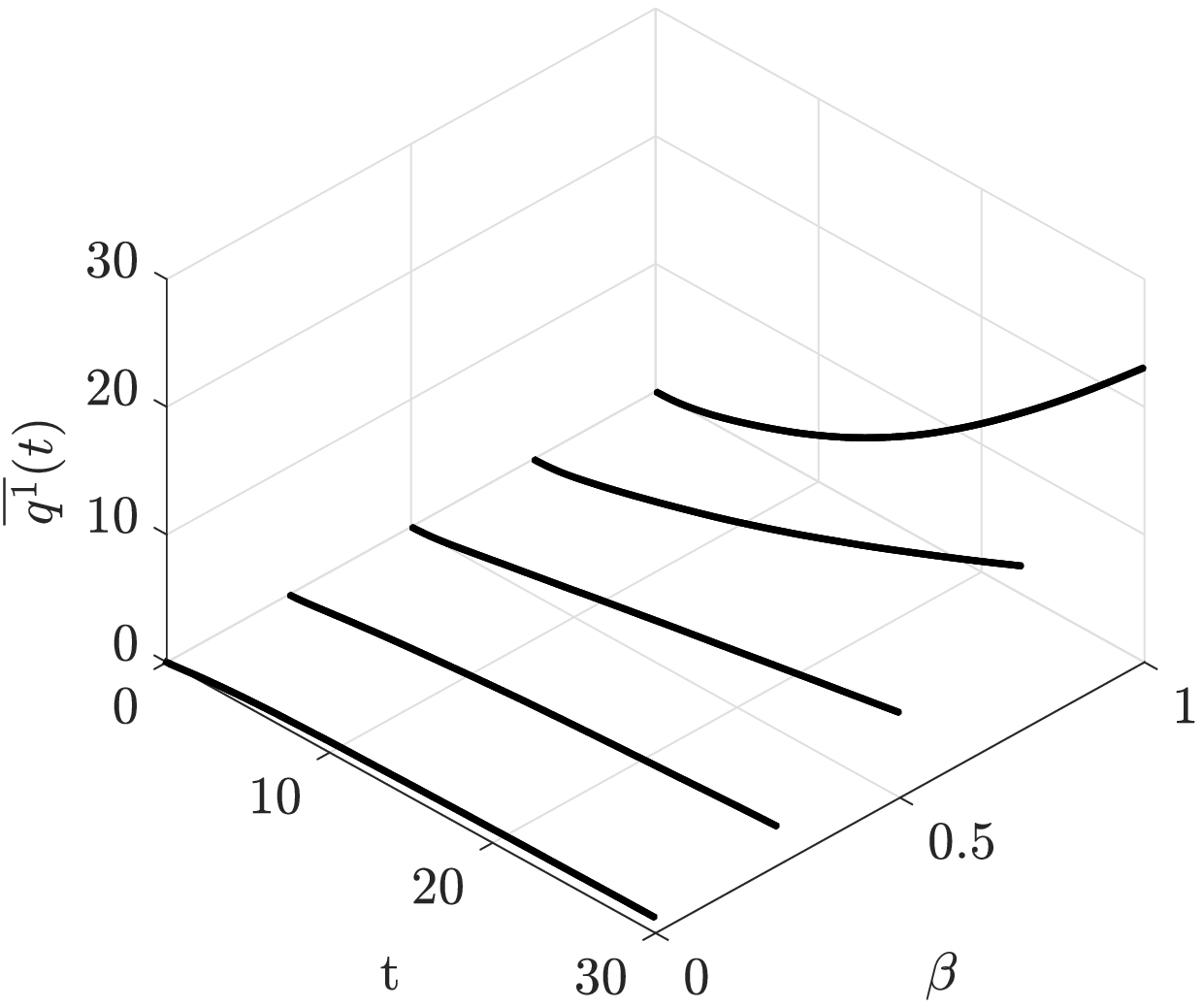}
\subcaption{Processor 1}
\end{subfigure}
\begin{subfigure}[c]{0.5\textwidth}
\centering
\includegraphics[width=0.9\textwidth]{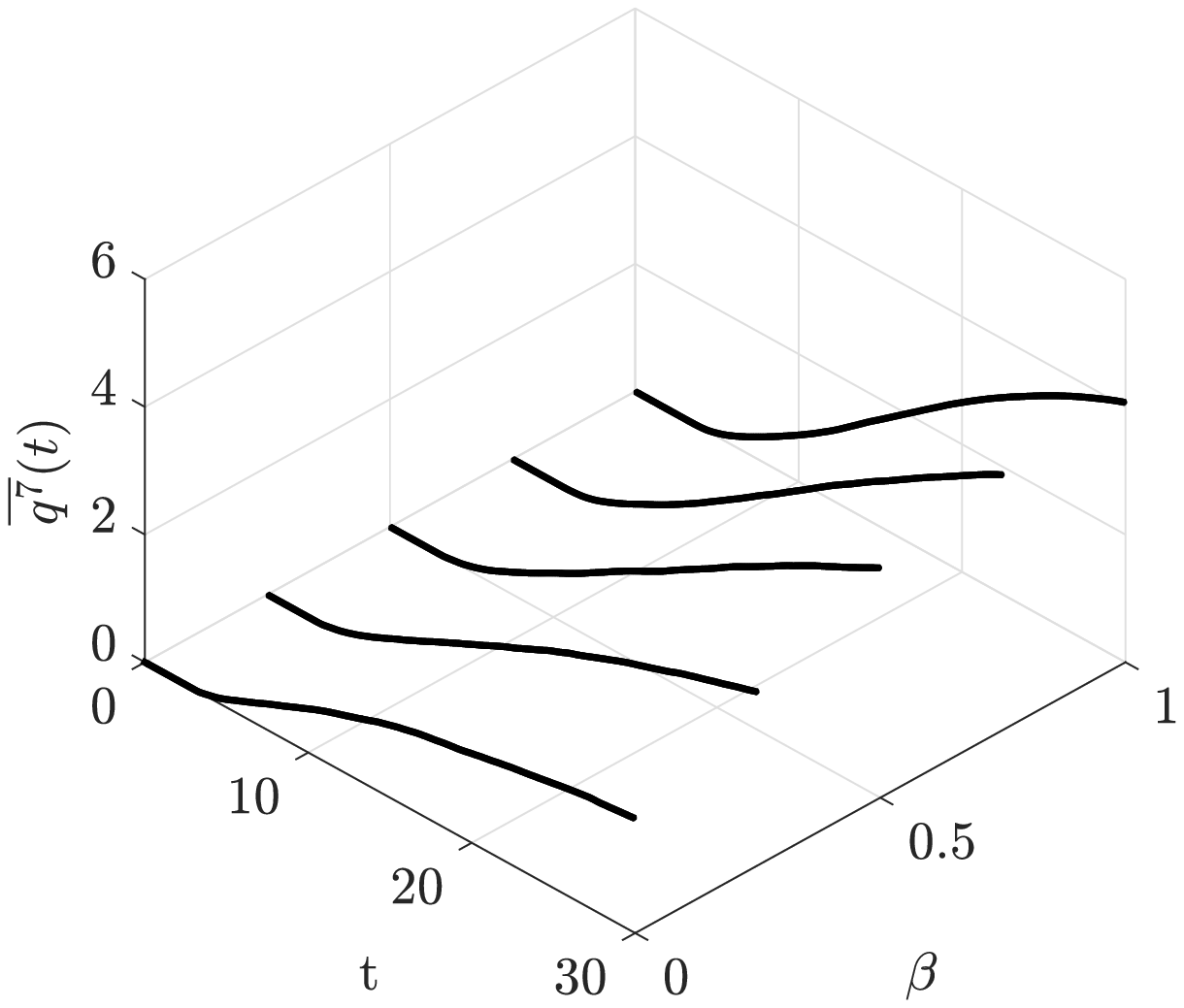}
\subcaption{Processor 7}
\end{subfigure}
\caption{Mean queue-lengths}
\label{fig:MeanQueueLoads1}
\end{figure}
If we evaluate the expected accumulated network queue-lengths, i.e.\ the sum of all queue-lengths until time $t=30$ and vary $\beta$, we see the significant influence of the load-dependency in figure \ref{fig:MeanNetworkQueuesAndOutflow1} (a). In the same manner, the expected accumulated network outflow decreases as $\beta$ increases, see figure \ref{fig:MeanNetworkQueuesAndOutflow1} (b).

The expected values contain no detailed information about the probability distribution of the accumulated network queue-lengths and outflow. Therefore, we consider a corresponding histogram in figure \ref{fig:NetworkQueuesAndOutflowHist} (a) and (b). On the $x$-axis, we have the parameter-value, on the $y$-axis the load-dependency $\beta$ and the color of the squares indicate the relative frequency. We deduce the increasing mean of the accumulated network queue-lengths again but we also observe the distortion of the probability distribution as $\beta$ changes. In detail, for $\beta = 0$ the distribution is more concentrated at the expected value than for $\beta>0$. Especially, the variance increases, which is reasonable by the following arguments: if a capacity drop happens, the probability that a capacity drop happens again after the repair is higher, since the processor is fed with the maximal capacity of the filled queue. On the other hand, if the machine is not affected by a capacity drop, the machine is emptied when the previous machine breaks down and consequently all subsequent machines are less affected by capacity drops. Summarizing together, the load-dependency increases the network dependency and consequently the variance of the network queue-lengths.
The same holds true for the network outflow, see figure \ref{fig:NetworkQueuesAndOutflowHist} (b).
\begin{figure}[H]
\begin{subfigure}[c]{0.5\textwidth}
\centering
\includegraphics[width=0.9\textwidth]{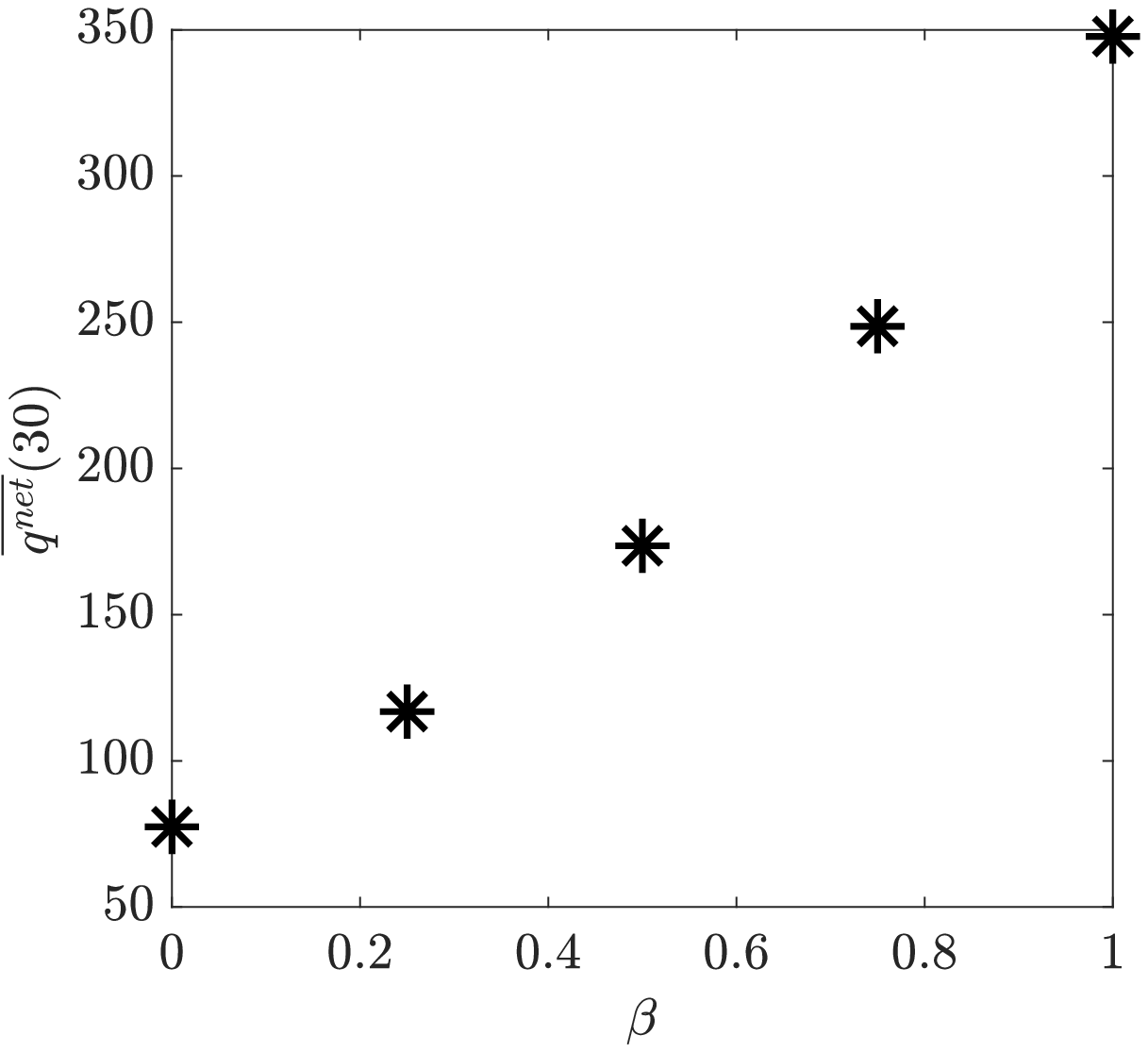}
\subcaption{Accumulated network queue-lengths}
\end{subfigure}
\begin{subfigure}[c]{0.5\textwidth}
\centering
\includegraphics[width=0.9\textwidth]{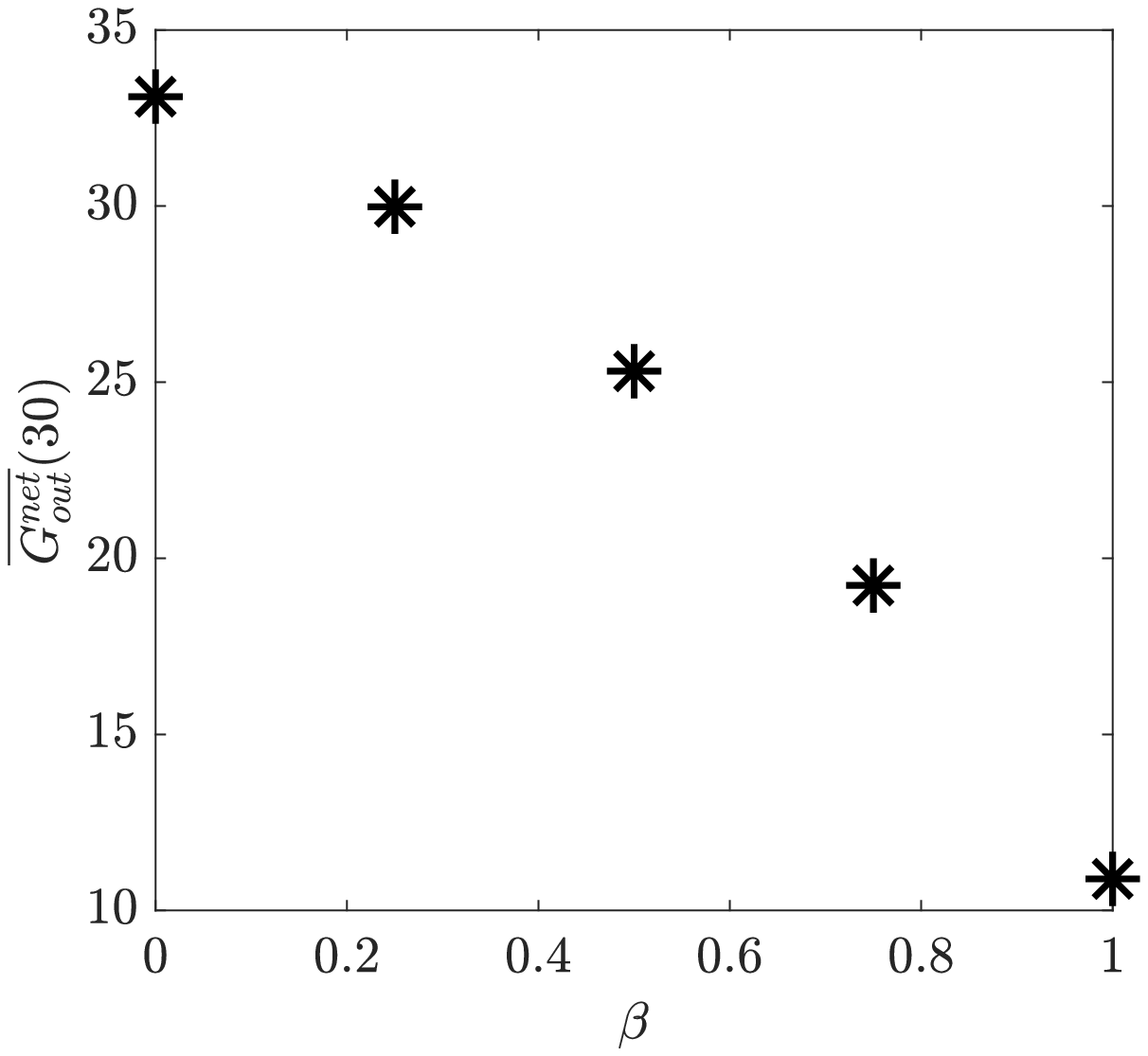}
\subcaption{Accumulated network outflow}
\end{subfigure}
\caption{Mean network measures}
\label{fig:MeanNetworkQueuesAndOutflow1}
\end{figure}
\begin{figure}[H]
\begin{subfigure}[c]{0.5\textwidth}
\centering
\includegraphics[width=0.9\textwidth]{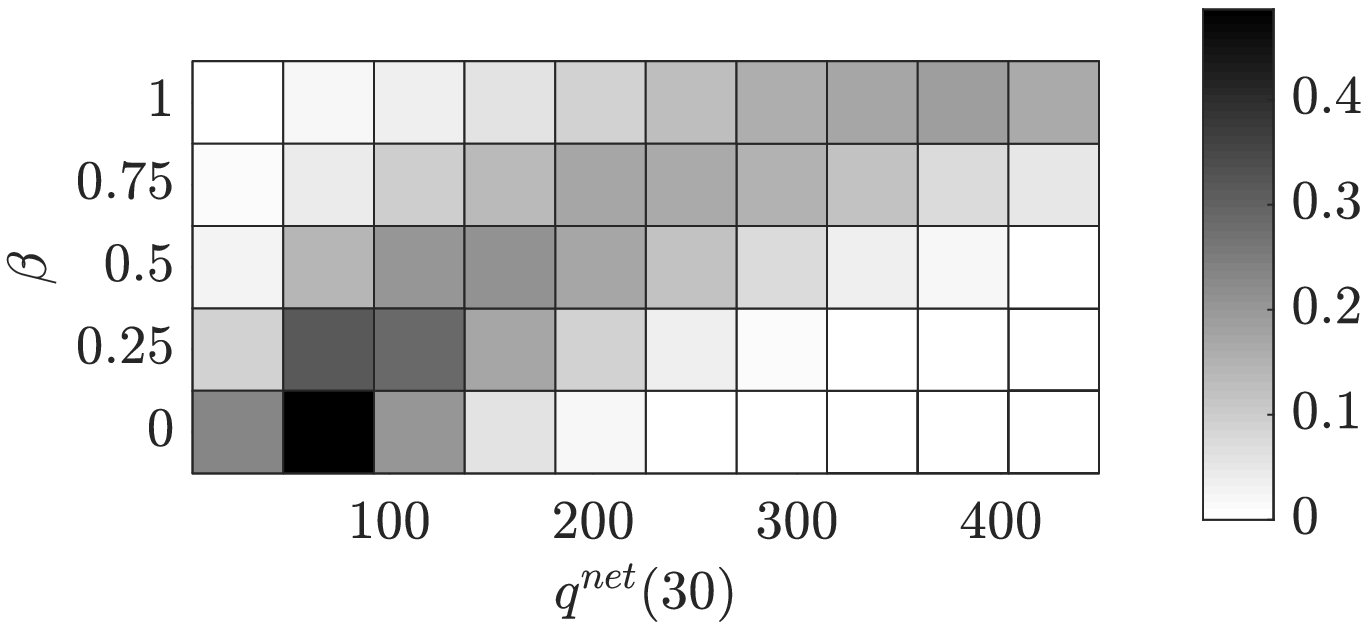}
\subcaption{Accumulated network queue-lengths}
\end{subfigure}
\begin{subfigure}[c]{0.5\textwidth}
\centering
\includegraphics[width=0.9\textwidth]{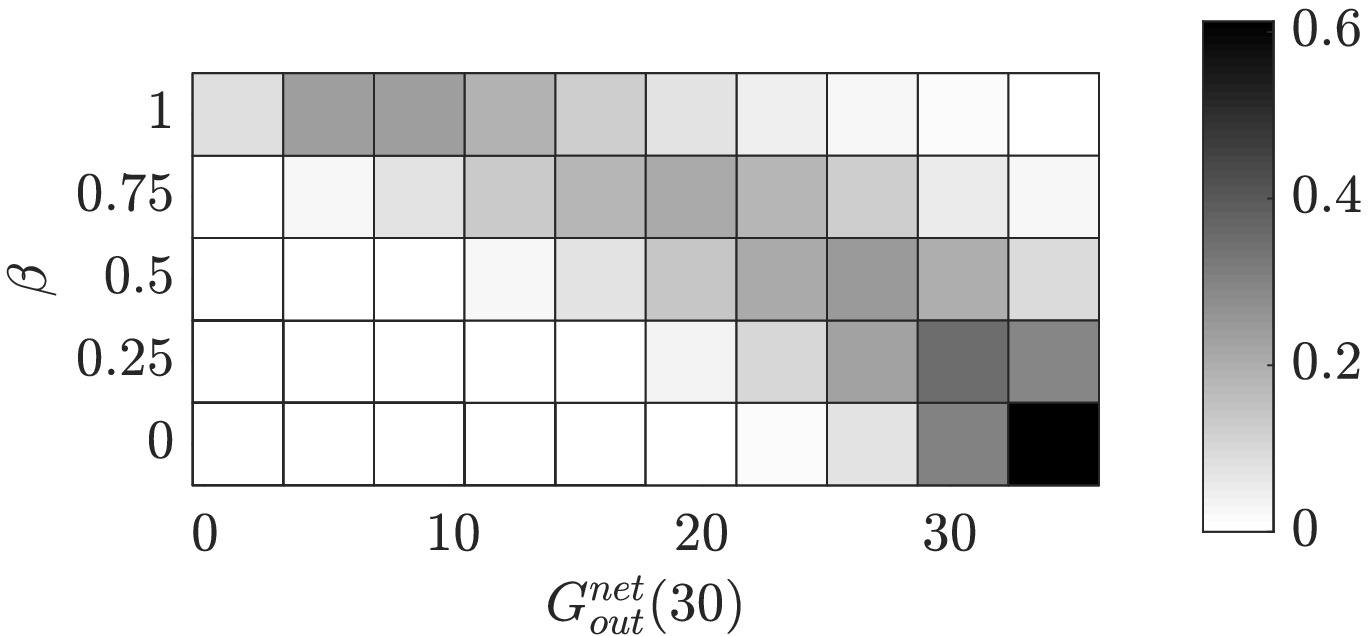}
\subcaption{Accumulated network outflow}
\end{subfigure}
\caption{Histogram of network measures}
\label{fig:NetworkQueuesAndOutflowHist}
\end{figure}
%

%% Inflow 2
In the following, we change the inflow function to $G_{\text{in}}^{1}(t) = 1 \cdot \Ind_{[0,5)}(t)+2 \cdot \Ind_{[10,15)}(t)+0.5 \cdot \Ind_{[20,25)}(t)$
to analyze the transient behavior of the load-dependent model numerically.
Figure \ref{fig:MeanCapacity2} shows the expected capacity of the first and last processor. An increase of the load-dependency $\beta$ reinforces the influence of the expected capacity on the network inflow as one would  expect. Processor 7 is less influenced by the varying inflow than processor 1, which implies that the network exhibits a smoothing effect.

If we set a high load-dependency $\beta = 1$, we observe in figure \ref{fig:MeanQueueLoads2} a high expected queue-length and consequently a stuck of products in the machine. This happens when the inflow jumps to the maximal capacity of 2 , which increases the probability of a capacity drop significantly. The first jump from 0 to 1 inflow causes no stuck in the processor and we conclude the importance of the inflow on the performance of the production network. 
\begin{figure}[H]
\begin{subfigure}[c]{0.5\textwidth}
\centering
\includegraphics[width=0.9\textwidth]{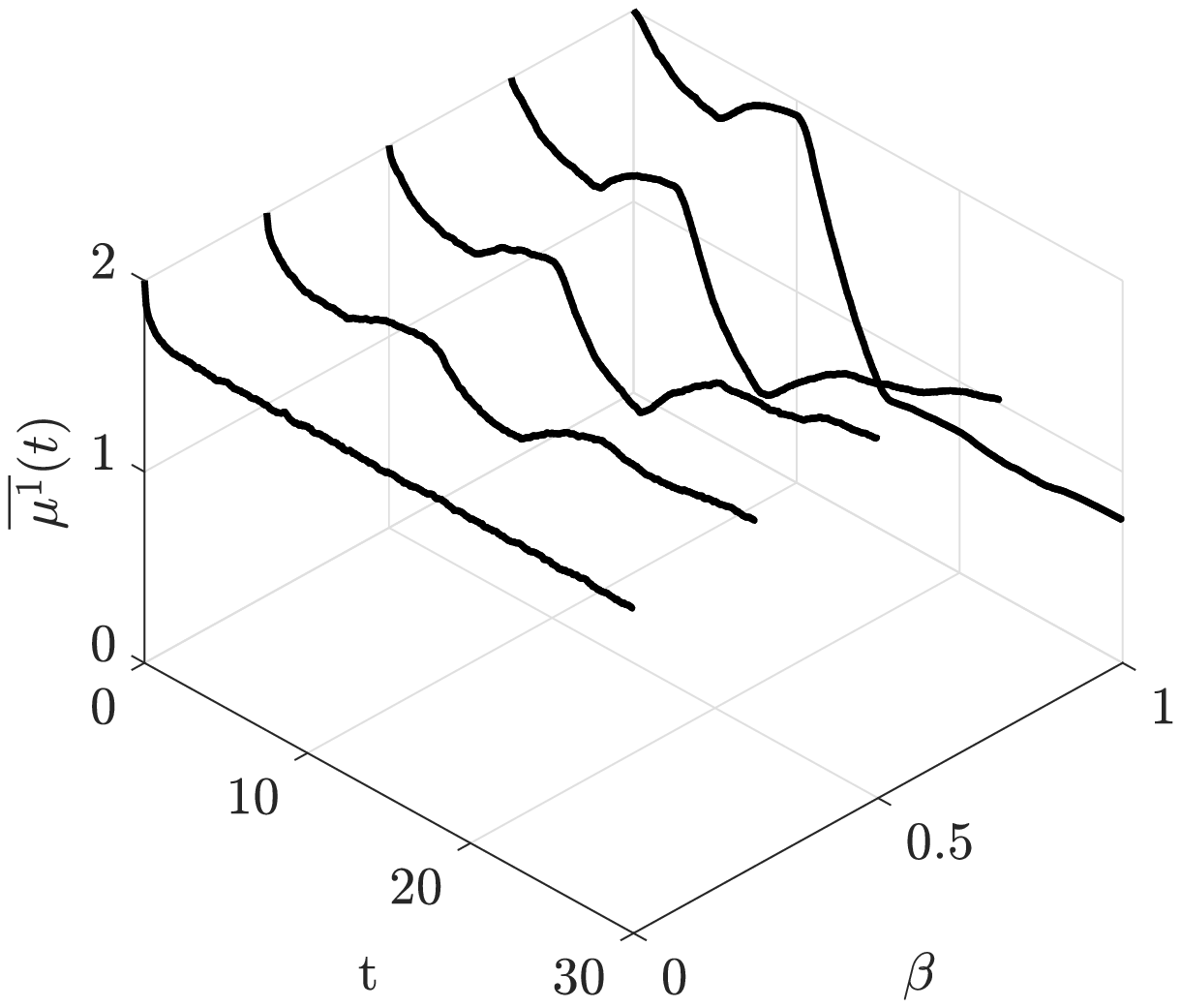}
\subcaption{Processor 1}
\end{subfigure}
\begin{subfigure}[c]{0.5\textwidth}
\centering
\includegraphics[width=0.9\textwidth]{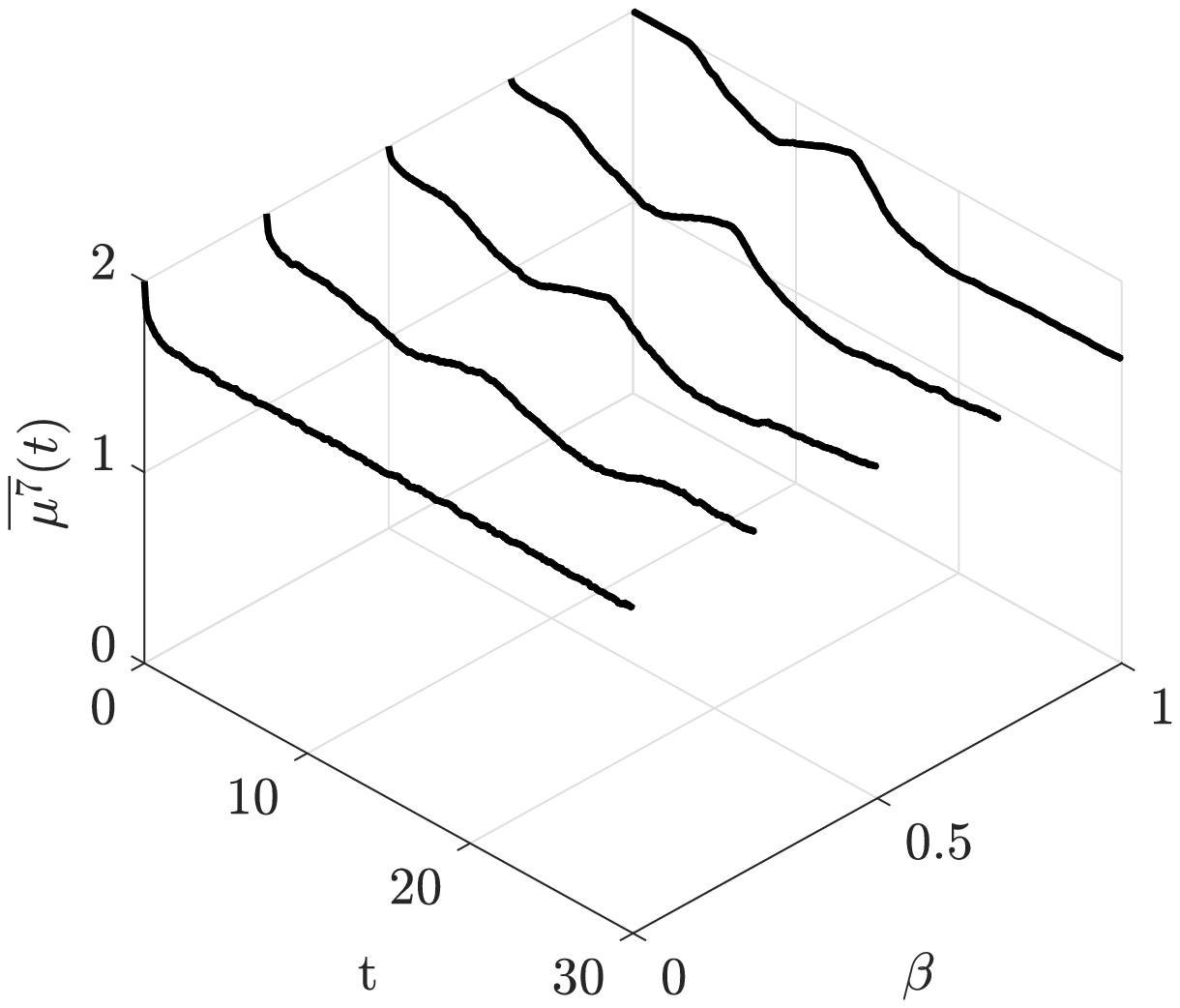}
\subcaption{Processor 7}
\end{subfigure}
\caption{Mean capacity}
\label{fig:MeanCapacity2}
\end{figure}
\begin{figure}[H]
\begin{subfigure}[c]{0.5\textwidth}
\centering
\includegraphics[width=0.9\textwidth]{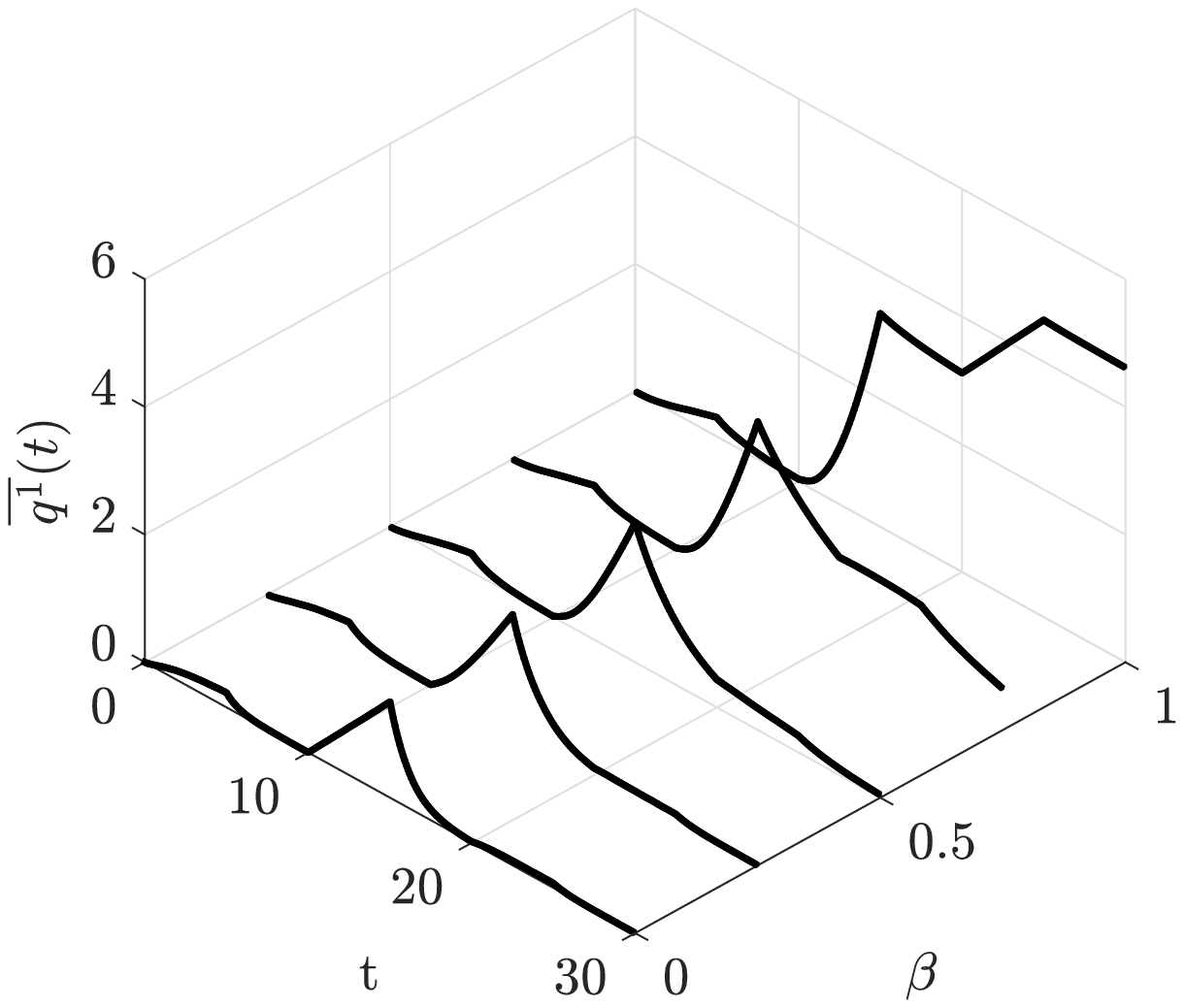}
\subcaption{Processor 1}
\end{subfigure}
\begin{subfigure}[c]{0.5\textwidth}
\centering
\includegraphics[width=0.9\textwidth]{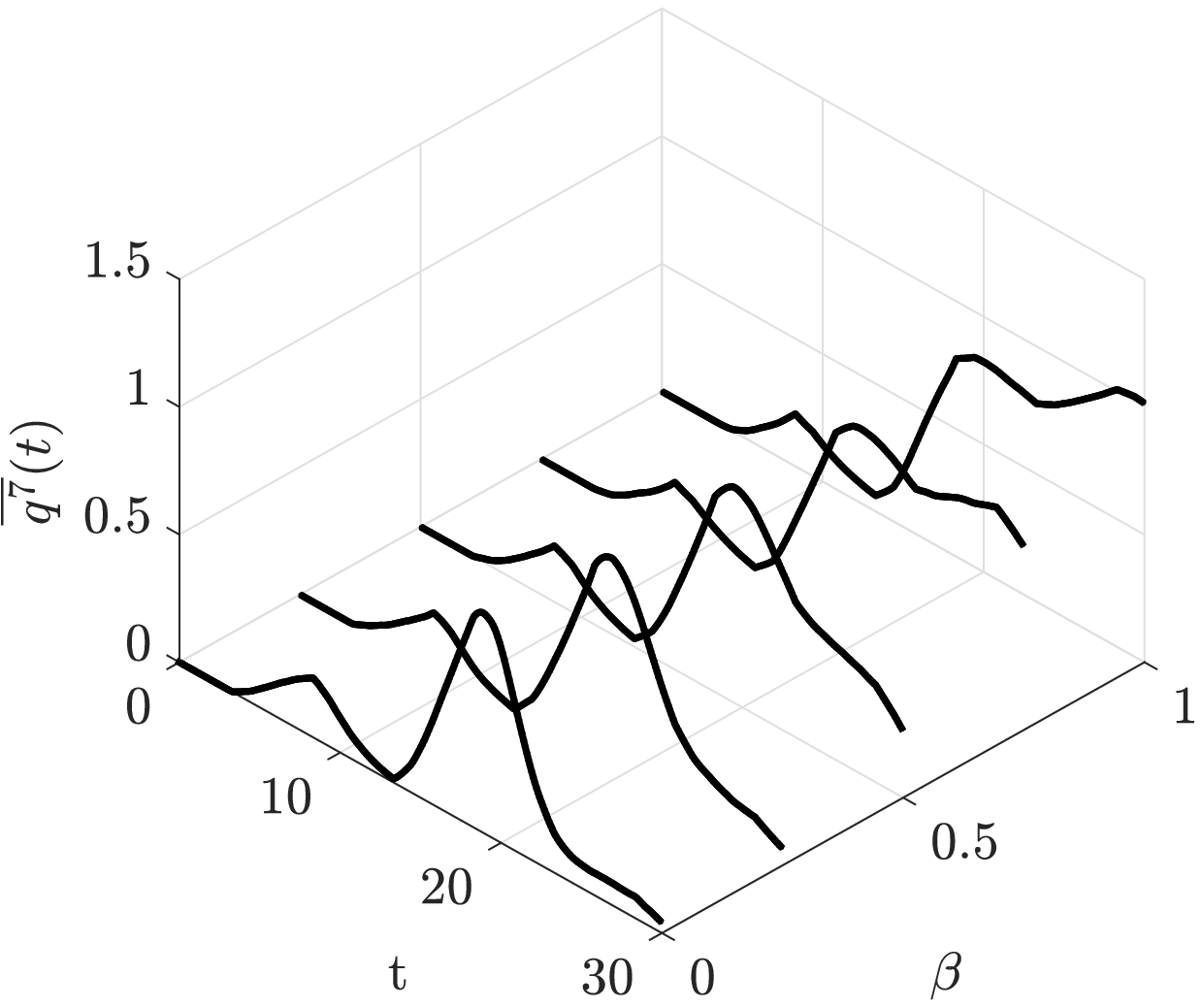}
\subcaption{Processor 7}
\end{subfigure}
\caption{Mean queue-lengths}
\label{fig:MeanQueueLoads2}
\end{figure}
In figure \ref{fig:NetworkQueuesAndOutflowHist2} a histogram of the accumulated network queue-lengths and outflow until time $t = 30$ is shown. The distribution for small $\beta$, i.e.\ $\beta \in \{0,0.25,0.5\}$ is concentrated around the expected value. If $\beta = 1$, the queue-lengths and the outflow distribution is affected by a high variance and consequently totally different and spread. This is exactly the case when the products get stuck in the production and we guess that there is a threshold $\beta = \beta^\ast$ at which the production network is 
unstable in some sense.
\begin{figure}[h]
\begin{subfigure}[c]{0.5\textwidth}
\centering
\includegraphics[width=0.9\textwidth]{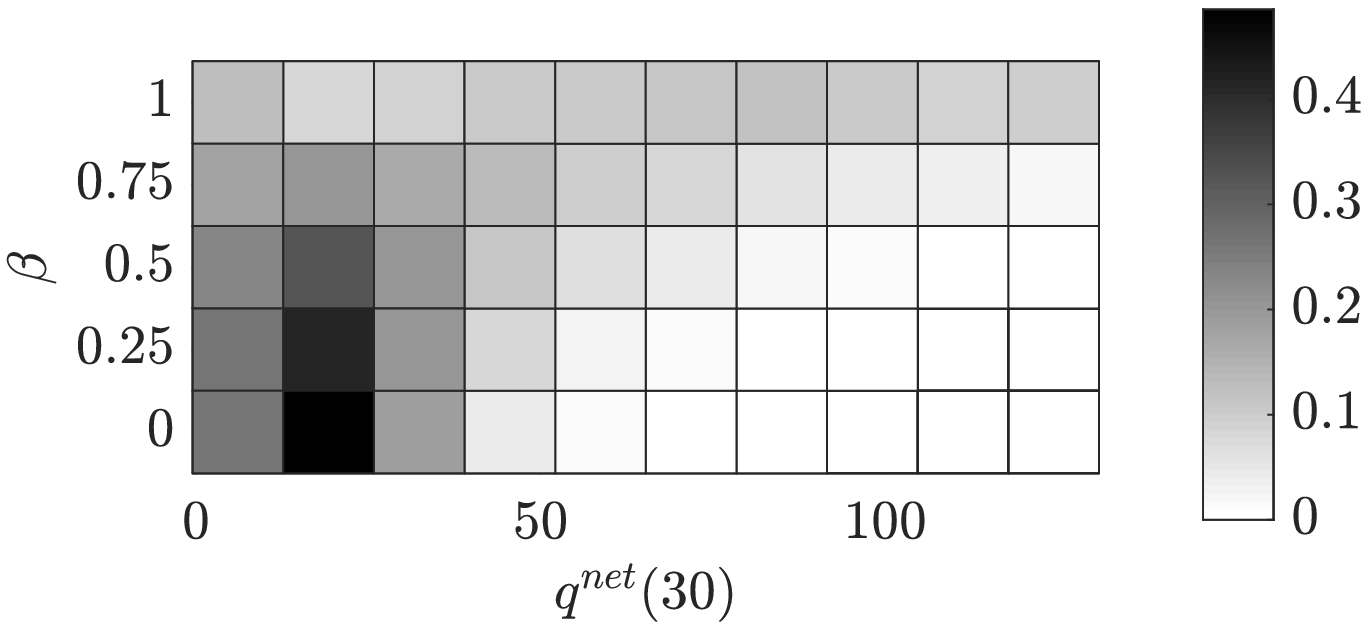}
\subcaption{Accumulated network queue-lengths}
\end{subfigure}
\begin{subfigure}[c]{0.5\textwidth}
\centering
\includegraphics[width=0.9\textwidth]{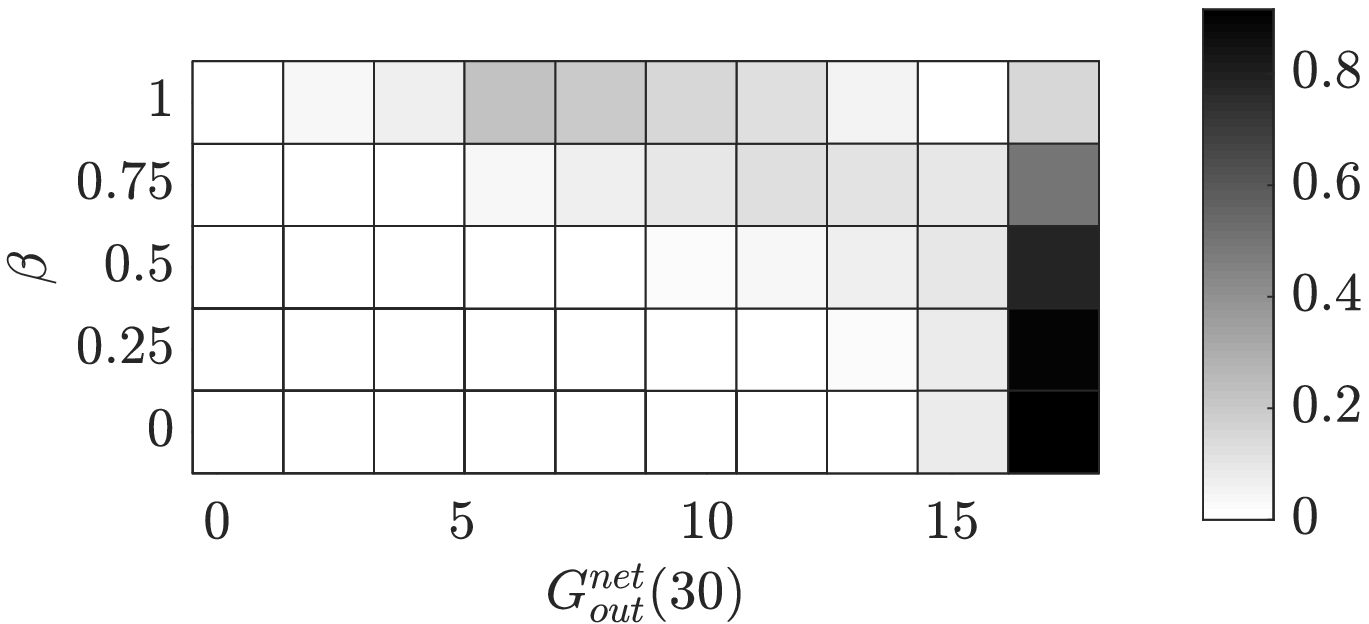}
\subcaption{Accumulated network outflow}
\end{subfigure}
\caption{Histogram of network measures}
\label{fig:NetworkQueuesAndOutflowHist2}
\end{figure}

\section{Conclusions}
\label{sec:conclusions}

We have introduced a load-dependent production network mo{\-}del, which deterministic dynamics is based on a system of coupled PDEs and ODEs. The stochastic effects have been implemented into the model by random capacity functions. Applying the theory of piecewise deterministic Markov processes, we have been able to show the existence of the model by considering an extended solution operator of the deterministic model. The construction of the stochastic production network model directly provides a stochastic simulation algorithm. Together with numerical approximation schemes for PDEs and ODEs, we can simulate sample paths and analyze the load-dependent model with a Monte-Carlo approach. 
The focus of the numerical study is the comparison of the load-independent and load-dependent model, where we observe a big influence on the expected queue-loads and network outflow. Also a distortion of the corresponding probability distributions emphasized the impact of the load-dependency.

We use a standard Monte-Carlo ansatz to evaluate performance measures of the production network model. This could be improved by more advanced Monte-Carlo techniques in future work. Additionally, alternative performance measures such as the profit could be introduced and investigated within an optimization framework.

\appendix
%\section{An example appendix} 

%\section*{Acknowledgments}
%We would like to acknowledge the assistance of volunteers in putting
%together this example manuscript and supplement.

\bibliographystyle{siamplain}
\bibliography{references}
\end{document}

%% file: ex_shared.tex
\usepackage{lipsum}
\usepackage{amsfonts}
\usepackage{graphicx}
\usepackage{epstopdf}
\usepackage{algorithmic}
\usepackage{amsmath}
\usepackage{amssymb}
\usepackage{amsfonts}
\usepackage{dsfont}
\ifpdf
  \DeclareGraphicsExtensions{.eps,.pdf,.png,.jpg}
\else
  \DeclareGraphicsExtensions{.eps}
\fi

\usepackage{subcaption}

\newcommand{\ga}{\alpha}

\newcommand{\gl}{\lambda}
\newcommand{\go}{\omega}

\newcommand{\gO}{\Omega}

\newcommand{\cA}{\mathcal{A}}
\newcommand{\cB}{\mathcal{B}}
\newcommand{\cC}{\mathcal{C}}

\newcommand{\cE}{\mathcal{E}}
\newcommand{\cF}{\mathcal{F}}

\newcommand{\cN}{\mathcal{N}}

\newcommand{\cV}{\mathcal{V}}

\newcommand{\cX}{\mathcal{X}}

\newcommand{\N}{\mathbb{N}}

\newcommand{\R}{\mathbb{R}}

\newcommand{\OAP}{$(\Omega,\mathcal{A},P)$}

\newcommand{\Ind}{\mathds{1}}

\newcommand{\sgn}{\operatorname{sgn}}

\newcommand{\lano}{\mathrm{o}}

\newcommand{\BV}{\operatorname{BV}}
\newcommand{\TV}{\operatorname{TV}}

\newcommand{\BIGOP}[1]{\mathop{\mathchoice%
{\raise-0.22em\hbox{\huge $#1$}}%
{\raise-0.05em\hbox{\Large $#1$}}{\hbox{\large $#1$}}{#1}}}
\newcommand{\bigtimes}{\BIGOP{\times}}
\newcommand{\BIGboxplus}{\mathop{\mathchoice%
{\raise-0.35em\hbox{\huge $\boxplus$}}%
{\raise-0.15em\hbox{\Large $\boxplus$}}{\hbox{\large $\boxplus$}}{\boxplus}}}

\newsiamremark{remark}{Remark}
\newsiamremark{hypothesis}{Hypothesis}
\crefname{hypothesis}{Hypothesis}{Hypotheses}
\newsiamthm{claim}{Claim}

\headers{Load-dependent machine failures in production network models}{S.~G\"ottlich, and S.~Knapp}

\title{Load-dependent machine failures in production network models\thanks{Submitted to the editors on June 8, 2018.
\funding{Financially supported by the BMBF project ENets (05M18VMA).}}}

\author{Simone G\"ottlich\thanks{Department of Mathematics, University of Mannheim, Mannheim, Germany 
  (\email{goettlich@uni-mannheim.de}, \email{stknapp@mail.uni-mannheim.de}).}
\and Stephan Knapp\footnotemark[2]}

\usepackage{amsopn}